\documentclass[reqno,10pt,centertags]{amsart}
\usepackage{amsmath,amsthm,amscd,amssymb,latexsym,esint,upref,stmaryrd,
enumerate,color,verbatim,yfonts}
\usepackage{hyperref}

\newcommand*{\mailto}[1]{\href{mailto:#1}{\nolinkurl{#1}}}

\newcommand{\bbC}{{\mathbb{C}}}

\newcommand{\bbN}{{\mathbb{N}}}

\newcommand{\bbR}{{\mathbb{R}}}

\newcommand{\cB}{{\mathcal B}}

\newcommand{\cH}{{\mathcal H}}

\newcommand{\cK}{{\mathcal K}}

\DeclareMathOperator{\supp}{supp}

\DeclareMathOperator{\dom}{dom}

\renewcommand{\Im}{\text{\rm Im}}

\newcommand{\no}{\notag}
\newcommand{\lb}{\label}
\newcommand{\f}{\frac}

\newcommand{\ol}{\overline}

\newcommand{\ti}{\tilde}

\newcommand{\oh}{o}
 
\newcommand{\bi}{\bibitem}
\newcommand{\dott}{\,\cdot\,}

\let\geq\geqslant
\let\leq\leqslant

\makeatletter
\def\theequation{\@arabic\c@equation}
\allowdisplaybreaks 
\numberwithin{equation}{section}

\newtheorem{theorem}{Theorem}[section]

\newtheorem{lemma}[theorem]{Lemma}

\newtheorem{definition}[theorem]{Definition}
\newtheorem{hypothesis}[theorem]{Hypothesis}
\newtheorem{example}[theorem]{Example}

\theoremstyle{remark}
\newtheorem{remark}[theorem]{Remark}

\begin{document}

\numberwithin{equation}{section}
\allowdisplaybreaks

\title[Principal Solutions Revisited]{Principal Solutions Revisited} 

\author[S.\ Clark]{Stephen Clark}
\address{Department of Mathematics \& Statistics,
Missouri University of Science and Technology, Rolla, MO 65409, USA}
\email{\mailto{sclark@mst.edu}}
%\email{sclark@mst.edu}
\urladdr{\url{http://web.mst.edu/\~sclark/}}
%\urladdr{http://web.mst.edu/\~{}sclark/}

\author[F.\ Gesztesy]{Fritz Gesztesy}
\address{Department of Mathematics,
University of Missouri, Columbia, MO 65211, USA}
\email{\mailto{gesztesyf@missouri.edu}}
%\email{gesztesyf@missouri.edu}
\urladdr{\url{http://www.math.missouri.edu/personnel/faculty/gesztesyf.html}}
%\urladdr{http://www.math.missouri.edu/personnel/faculty/gesztesyf.html}

\author[R.\ Nichols]{Roger Nichols}
\address{Mathematics Department, The University of Tennessee at Chattanooga, 415 EMCS Building, Dept. 6956, 615 McCallie Ave, Chattanooga, TN 37403, USA}
\email{\mailto{Roger-Nichols@utc.edu}}
%\email{Roger-Nichols@utc.edu}
\urladdr{\url{http://www.utc.edu/faculty/roger-nichols/}} 
%\urladdr{http://www.utc.edu/faculty/roger-nichols/}

\dedicatory{Dedicated with admiration to Ludwig Streit on the occasion of his 75th birthday}

\date{\today}
\thanks{R.N. gratefully acknowledges support from an AMS--Simons Travel Grant.}
\thanks{In {\it Stochastic and Infinite Dimensional Analysis}, C.\ C.\ Bernido, 
M.\ V.\ Carpio-Bernido, M.\ Grothaus, T.\ Kuna, M.\ J.\ Oliveira, and J. L. da Silva (eds.), 
Trends in Mathematics, Birkh\"auser, to appear.}

%%%%%%%%%%%%
%%%%%%%%%%%%
\begin{abstract} 
The main objective of this paper is to identify principal solutions associated with 
Sturm--Liouville operators on arbitrary open intervals $(a,b) \subseteq \bbR$, as 
introduced by Leighton and Morse in the scalar context in 1936 and by Hartman 
in the matrix-valued situation in 1957, with Weyl--Titchmarsh solutions, as 
long as the underlying Sturm--Liouville differential expression is nonoscillatory (resp., 
disconjugate or bounded from below near an endpoint) and in the limit point case at 
the endpoint in question. In addition, we derive an explicit formula for 
Weyl--Titchmarsh functions in this case (the latter appears to be new in the matrix-valued 
context). 
\end{abstract}
%%%%%%%%%%%%
%%%%%%%%%%%%

\maketitle 

%\vspace*{-3mm} 
%{\scriptsize{\tableofcontents}}
%\normalsize

%%%%%%%%%%%%%%%%%%%%%%%%%%%%%%%%%%%%%%%%%%
%%%%%%%%%%%%%%%%%%%%%%%%%%%%%%%%%%%%%%%%%%
\section{Introduction}  \lb{s1}
%%%%%%%%%%%%%%%%%%%%%%%%%%%%%%%%%%%%%%%%%%
%%%%%%%%%%%%%%%%%%%%%%%%%%%%%%%%%%%%%%%%%%

We dedicate this paper to Ludwig Streit in great appreciation of the tremendous influence 
he exerted on all those who were permitted a glimpse at his boundless curiosity and approach to all   
aspects of science. We hope this modest contribution will create some joy for him. 

The main focus of this paper centers around principal and Weyl--Titchmarsh solutions 
for general Sturm--Liouville operators (associated with three coefficients) on arbitrary 
open intervals $(a,b) \subseteq \bbR$. We will discuss in great detail the case 
of scalar coefficients $p, q, r$ associated with the differential expression 
\begin{equation}
\ell =\f{1}{r}\bigg(-\f{d}{dx}p\f{d}{dx}+q\bigg), \quad 
-\infty\leq a<x<b\leq\infty,   \lb{1.1}
\end{equation}
and corresponding operator realizations in the Hilbert space $L^2((a,b);rdx)$, as well as 
the case of $m \times m$ matrix-valued coefficients $P, Q, R$, $m \in \bbN$, associated with the 
differential expresssion
\begin{equation}
L = R^{-1}\bigg(-\f{d}{dx}P\f{d}{dx} + Q\bigg), \quad 
-\infty\leq a<x<b\leq\infty,   \lb{1.2}
\end{equation}
and corresponding operator realizations in the Hilbert space $L^2((a,b);R dx; \bbC^m)$.

Focusing in this introduction for reasons of brevity exclusively on the right end point $b$, 
if $\ell$ is nonoscillatory at $b$, (real-valued) {\it principal solutions} $u_b(\lambda, \cdot \,)$ of 
$\ell u = \lambda u$, $\lambda \in \bbR$, are characterized by the condition that 
$u_b(\lambda, \cdot \,)$ does not vanish in a neighborhood $[c,b)$ of $b$ (with $c \in (a,b)$) 
and that
\begin{equation}
\int_c^b dx \, p(x)^{-1}u_b(\lambda,x)^{-2}=\infty.      \lb{1.3}
\end{equation}
As discussed in Lemma \ref{l2.7}, $u_b(\lambda, \cdot \,)$ is unique up to constant (possibly, 
$\lambda$-dependent) multiples and, in a certain sense (made precise in Lemma \ref{l2.7}), 
also characterized as the smallest (minimal) possible solution of $\ell u = \lambda u$ near 
the endpoint $b$.

In contrast to \eqref{1.3}, if $\ell$ is in the limit point case at $b$, {\it Weyl--Titchmarsh solutions} 
$\psi_+ (z, \cdot \,)$ of $\ell u = z u$, $z \in \bbC \backslash \bbR$, 
are characterized by the condition that for some (and hence for all) $c \in (a,b)$,
\begin{equation}
\psi_+ (z, \cdot \,) \in L^2((c,b);rdx) \quad z \in \bbC \backslash \bbR.      \lb{1.4} 
\end{equation} 
Again, $\psi_+ (z, \cdot \,)$ is unique up 
to constant (generally, $z$-dependent) multiples.

Our main result, Theorem \ref{t2.13} in Section \ref{s2}, then proves equality of these solutions 
(up to constant, possibly spectral parameter dependent multiples) under appropriate 
assumptions. More precisely, assuming $\ell$ to be nonoscillatory and in the limit point case 
at $b$, there exists $\lambda_b \in \bbR$, such that for all $\lambda < \lambda_b$, 
$x,x_0 \in (a,b)$, with $x, x_0$ beyond the last zero of 
$\psi_+ (\lambda, \cdot \,), u_b(\lambda,\cdot \,)$ (if any),  
\begin{equation}
\psi_+(\lambda,x) \psi_+ (\lambda, x_0)^{-1} = u_b(\lambda,x) u_b(\lambda,x_0)^{-1}.    \lb{1.5}
\end{equation}
Here, $\psi_+(\lambda,\,\cdot\,)$, $\lambda<\lambda_b$, denotes the extension of $\psi_+(z,\,\cdot\,)$, defined initially only for $z\in \bbC\backslash \bbR$, to real values $z<\lambda_b$.  This extension is permitted on the basis that $\ell$ is assumed to be nonoscillatory and in the limit point case at $b$  (cf.~Remark \ref{r2.12}).

We also recall Green's function formulas in terms of principal solutions and an explicit formula 
for the Weyl--Titchmarsh function at the end of Section \ref{s2}, supposing the underlying limit 
point assumptions on $\ell$. 

In Section \ref{s3}, the main new section in this paper, we prove the analogous results in the 
matrix-valued setting. We will be 
primarily concerned with self-conjugate solutions $U(\lambda, \cdot \,)$ of 
$LU = \lambda U$, $\lambda \in \bbR$, defined by the vanishing of the underlying 
$m \times m$ matrix-valued Wronskian,
\begin{equation}
W(U(\lambda, \cdot \,)^*, U(\lambda, \cdot \,)) = 0,   \quad \lambda \in \bbR.    \lb{1.6} 
\end{equation}
Focusing again exclusively on the endpoint $b$, a self-conjugate solution $U_b(\lambda, \cdot \,)$ 
of $L U = \lambda U$ that is invertible on $[c,b)$ for some $c \in (a,b)$ is called a 
{\it principal solution} of $L U = \lambda U$ at $b$ if 
\begin{equation}
\lim_{x \uparrow b} \bigg[\int_c^x dx' \, 
U_b(\lambda,x')^{-1} P(x')^{-1} \big[U_b(\lambda,x')^{-1}\big]^*\bigg]^{-1} = 0.    \lb{1.7}
\end{equation}   
Again, by Lemma \ref{l3.5}, $U_b(\lambda, \cdot \,)$ is unique up to right multiplication by 
invertible (possibly, $\lambda$-dependent) constant $m \times m$ matrices, and in a certain 
sense (detailed in Lemma \ref{l3.6}) it represents the smallest (minimal) solution of 
$LU = \lambda U$ near the endpoint $b$.

In analogy to \eqref{1.4}, if $L$ is in the limit point case at $b$, {\it Weyl--Titchmarsh solutions}  
$\Psi_+ (z, \cdot \,)$ of $L U = z U$, $z \in \bbC \backslash \bbR$, are then 
characterized by the condition that for some (and hence for all) $c \in (a,b)$, there exists an  
invertible $m \times m$ matrix-valued solution 
$\Psi_+(z, \cdot \,)$ of $L U = z U$ such that the $m \times m$ 
matrices 
\begin{equation}
\int_c^b dx \,  \Psi_+ (z,x)^* R(x) \Psi_+ (z,x), \quad  z \in \bbC \backslash \bbR,    \lb{1.8} 
\end{equation} 
exist. As in the context of principal solutions, $\Psi_+(z, \cdot \,)$ is unique up to right 
multiplication by (generally, $z$-dependent) invertible $m \times m$ matrices and it can be shown 
that $\Psi_+(z, \cdot \,)$ is self-conjugate. 

Our main result, Theorem \ref{t3.10} in Section \ref{s3}, once again proves equality of these 
solutions (up to right multiplication by possibly, spectral parameter dependent invertible $m \times m$ 
matrices) under appropriate assumptions. More precisely, assuming the existence of 
$\lambda_b \in \bbR$, such that $L - \lambda_b I$ is disconjugate on $[c,b)$ for all $c \in (a,b)$, 
and supposing $L$ to be in the limit point case at $b$, then for all $\lambda < \lambda_b$, 
$x,x_0 \in (a,b)$, with $x, x_0$ beyond the last zero of 
$\det_{\bbC^m}(\Psi_+ (\lambda, \cdot \,)), \det_{\bbC^m}(U_b(\lambda,\cdot \,))$ (if any),  
\begin{equation}
\Psi_+(\lambda,x) \Psi_+ (\lambda, x_0)^{-1} = U_b(\lambda,x) U_b(\lambda,x_0)^{-1}.    \lb{1.9}
\end{equation}

In addition, with the normalized $m \times m$ matrix-valued solutions $\Theta(z, \cdot \, , x_0)$ 
of $L U = z U$ defined by 
\begin{equation}
\Theta(z, x_0 , x_0) = I_m, \quad [P(x) \Theta'(z, x, x_0)](x_0)|_{x =x_0} = 0,   \lb{1.10} 
\end{equation}
we will show the following formula for the $m \times m$ matrix-valued Weyl--Titchmarsh function 
associated with $L$,
\begin{equation}
M_+(z,x_0) = - \lim_{x \uparrow b} \bigg[\int_{x_0}^x dx' \, \Theta(z,x',x_0)^{-1} P(x')^{-1} 
\big[\Theta({\ol z},x',x_0)^{-1}\big]^*\bigg]^{-1}, \quad z \in \bbC \backslash \bbR,    \lb{1.11} 
\end{equation}
assuming $L$ to be in the limit point case at $b$. If in addition, $Lu = \lambda_b u$ is 
disconjugate for some $\lambda_b \in \bbR$, then also 
\begin{equation}
M_+(\lambda,x_0) = - \bigg[\int_{x_0}^b dx' \, \Theta(\lambda,x',x_0)^{-1} P(x')^{-1} 
\big[\Theta(\lambda,x',x_0)^{-1}\big]^*\bigg]^{-1}, \quad \lambda < \lambda_b,     \lb{1.12} 
\end{equation} 
holds, and  
\begin{align} 
& \big(\xi, M_+(\lambda,x_0)^{-1} \eta\big)_{\bbC^m}= - \int_{x_0}^b dx' \, 
\big(\xi, \Theta(\lambda,x',x_0)^{-1} P(x')^{-1} 
\big[\Theta(\lambda,x',x_0)^{-1}\big]^* \eta\big)_{\bbC^m},   \no \\ 
& \hspace*{8.2cm} \lambda < \lambda_b, \; \xi, \eta \in \bbC^m,    \lb{1.13} 
\end{align}  
exists as a Lebesgue integral. Both formulas, \eqref{1.11} and \eqref{1.12}, are of independent 
interest and we know of no previous source that recorded them. 

Concluding this introduction, we briefly summarize some of the notation used in this paper. 
If $\cH$ is a separable complex Hilbert space the symbol $(\dott,\dott)_{\cH}$ denotes the 
scalar product in $\cH$ (linear in the second entry). If $T$ is a linear operator 
mapping (a subspace of) a Hilbert space into another, $\dom(T)$ denotes the
domain of $T$. The spectrum and resolvent set of a closed linear operator 
in $\cH$ will be denoted by $\sigma(\cdot)$ and $\rho(\cdot)$, respectively. The closure of 
a closable operator $S$ in $\cH$ is denoted by $\ol S$. 

The Banach spaces of bounded and compact linear operators on $\cH$ are
denoted by $\cB(\cH)$ and $\cB_\infty(\cH)$, respectively.   

The symbol $I_m$, $m \in \bbN$, represents the identity operator in $\bbC^m$. 
The set of $m \times m$ matrices with complex-valued (resp., real-valued) entries is 
abbreviated by $\bbC^{m \times m}$ (resp., $\bbR^{m \times m}$), and similarly, 
$L^s ((c,d); dx)^{m \times m}$ (resp., $L^s_{loc} ((c,d); dx)^{m \times m}$) denotes the set 
of $m \times m$ matrices with entries in $L^s((c,d); dx)$ (resp., $L^s_{loc}((c,d); dx)$), where 
$s>0$ and $a \leq c < d \leq b$. For notational simplicity, $I$ represents the identity operator in $L^2((a,b);rdx)$ and also in $L^2((a,b); rdx; \bbC^m)$. 

Finally, $\bbC_+$ (resp., $\bbC_-$) denotes the open complex upper (resp., lower) half-plane, and we will use the abbreviation ``a.e.'' for ``Lebesgue almost everywhere.''

%%%%%%%%%%%%%%%%%%%%%%%%%%%%%%
%%%%%%%%%%%%%%%%%%%%%%%%%%%%%%
\section{Basic Facts on Scalar Principal Solutions} \lb{s2}
%%%%%%%%%%%%%%%%%%%%%%%%%%%%%%
%%%%%%%%%%%%%%%%%%%%%%%%%%%%%%

In this preparatory section we recall some of the basic facts on oscillation theory 
with particular emphasis on principal solutions, a notion originally due to Leighton and Morse 
\cite{LM36}, in connection with scalar Sturm--Liouville operators on arbitrary open intervals 
$(a,b) \subseteq \bbR$.  

We start by summarizing a few key results in the one-dimensional
scalar case, whose extension to the matrix-valued context we are
particularly interested in.

Our basic hypothesis in this section will be the following (however, we emphasize that 
all results in this section have been proved under more general conditions on the 
coefficients $p, q$, and for more general differential expressions $\ell$, in \cite{EGNT13}).

%%%%%%%%%%%%%
\begin{hypothesis} \lb{h2.1} 
Let $-\infty\leq a<b\leq\infty$ and suppose that $p,q,r$ are
$($Lebesgue$\,)$ measurable on $(a,b)$, and that 
\begin{align}
\begin{split} 
& p>0, r>0 \text{ a.e.\ on $(a,b)$, $q$ is real-valued},     \\
& 1/p, q, r \in L^1_{loc} ((a,b); dx).      \lb{2.1}
\end{split} 
\end{align}
\end{hypothesis}
%%%%%%%%%%%%%

Given Hypothesis \ref{h2.1}, we consider the differential expression
\begin{equation}
\ell =\f{1}{r}\bigg(-\f{d}{dx}p\f{d}{dx}+q\bigg), \quad 
-\infty\leq a<x<b\leq\infty,   \lb{2.2}
\end{equation}
and define
the minimal operator $T_{min}$ and maximal operator $T_{max}$ in $L^2((a,b);rdx)$ 
associated with $\ell$ by 
\begin{align}
& T_{min}u=\ell u,   \no \\
& u\in \dom (T_{min}) = \big\{v \in L^2((a,b);rdx)\,\big|\, v, 
pv' \in AC_{loc} ((a,b));     \lb{2.5} \\
& \hspace*{3cm}  \supp\,(v)\subset (a,b) \, \text{compact}; \, \ell v \in L^2((a,b);rdx)\big\},  \no \\
& T_{max}u=\ell u,    \lb{2.6} \\ 
& u\in \dom (T_{max})= \big\{v \in L^2((a,b);rdx)\,\big|\, v, 
pv'\in AC_{loc} ((a,b)); \, \ell v \in L^2((a,b);rdx)\big\},  \no 
\end{align}
respectively. Here $AC_{loc}((a,b))$ denotes the set of locally absolutely continuous
functions on $(a,b)$.  

Then $T_{min}$ is densely defined and 
\cite[p.\ 64, 88]{Na68}
\begin{equation}
{T_{min}}^*=T_{max}, \quad T_{max}^* = \ol{T_{min}}. \lb{2.7}
\end{equation}

%%%%%%%%%%%%%%
\begin{remark}\lb{r2.2} 
$(i)$ In obvious notation, we will occasionally write $[p(x_0) u'(x_0)]$ for the quasi-derivative 
$p u'|_{x = x_0}$. \\[.5mm] 
$(ii)$ In the following we will frequently invoke 
solutions $u(z,\cdot\,)$ of $\ell u = z u$ for some $z\in\bbC$. Such 
solutions are always assumed to be distributional solutions, that is, we
tacitly assume 
\begin{equation}
u(z,\cdot\,), \, p(\,\cdot\,)u'(z,\cdot\,)\in AC_{loc} ((a,b)) \lb{2.7a}
\end{equation}
in such a case. 
\end{remark}
%%%%%%%%%%%%%%

%%%%%%%%%%%%%%
\begin{lemma} [cf., e.g., \cite{GZ91}] \lb{l2.3} 
Assume Hypothesis \ref{h2.1}. \\[1mm]
$(i)$ Suppose $\ell u =\lambda u$ for some $\lambda\in\bbR$ 
with $u(\lambda,\cdot \,) \geq 0$ $($$u(\lambda,\cdot \,)\not\equiv 0$$)$ on $(a,b)$. Then 
$u(\lambda,\cdot \,) > 0$ on 
$(a,b)$. \\[1mm]
$(ii)$ $($Harnack's inequality$\,)$. Let $\cK\subset (a,b)$ be compact and
$\lambda\in\bbR$.  Then there exists a $C_{\cK,\lambda} >0$ such that
for all solutions $u(\lambda,\cdot \,)\geq 0$ satisfying $\ell u = \lambda u$, one has
\begin{equation}
\sup_{x\in\cK}(u(\lambda,x))\leq C_{\cK,\lambda} \inf_{x\in\cK}(u(\lambda,x)).    \lb{2.8}
\end{equation}
\end{lemma}
%%%%%%%%%%%%%%

%%%%%%%%%%%%%%
\begin{definition} \lb{d2.4}
Assume Hypothesis \ref{h2.1}. \\[1mm] 
$(i)$ Fix $c\in (a,b)$. Then $\ell$ is
called {\it nonoscillatory} near $a$ $($resp., $b$$)$ for some $\lambda\in\bbR$
if and only if every solution $u(\lambda,\cdot \,)$ of 
$\ell u = \lambda u$ has finitely many
zeros in $(a,c)$ $($resp., $(c,b)$$)$. Otherwise, $\ell$ is called {\it oscillatory}
near $a$ $($resp., $b$$)$. \\[1mm] 
$(ii)$ Let $\lambda_0 \in \bbR$. Then $T_{min}$ is bounded from below by $\lambda_0$, 
and one writes $T_{min} \geq \lambda_0 I$, if 
\begin{equation} 
(u, [T_{min} - \lambda_0 I]u)_{L^2((a,b);rdx)}\geq 0, \quad u \in \dom(T_{min}).
\end{equation}
\end{definition}
%%%%%%%%%%%%%%

The following is a key result. 

%%%%%%%%%%%%%%
\begin{theorem} [\cite{Ha82}, \cite{Ka78}, \cite{Re51}, 
\cite{Sc01}] \lb{t2.5} 
Assume Hypothesis \ref{h2.1}. Then the following assertions are
equivalent: \\[1mm] 
$(i)$ $T_{min}$ $($and hence any symmetric extension of $T_{min}\,)$
is bounded from below. \\[1mm] 
$(ii)$ There exists a $\lambda_0\in\bbR$ such that $\ell$ is
nonoscillatory near $a$ and $b$ for all $\lambda < \lambda_0$. \\[1mm]
$(iii)$ For fixed $c \in (a,b)$, there exists a $\lambda_0\in\bbR$ such that for all
$\lambda<\lambda_0$, $\ell u = \lambda u$ has solutions
$u_a(\lambda,\cdot\,)>0$,
$\hat u_a(\lambda,\cdot\,)>0$ in a neighborhood $(a,c]$ of $a$, and solutions
$u_b(\lambda,\cdot\,)>0$, $\hat u_b(\lambda,\cdot\,)>0$ in a neighborhood $[c,b)$ of
$b$, such that 
\begin{align}
&W(u_a (\lambda,\cdot\,),\hat u_a (\lambda,\cdot\,))=1,
\quad u_a (\lambda,x)=\oh(\hat u_a (\lambda,x))
\text{ as $x\downarrow a$,} \lb{2.9} \\
&W(u_b (\lambda,\cdot\,),\hat u_b (\lambda,\cdot\,))\,=1,
\quad u_b (\lambda,x)\,=\oh(\hat u_b (\lambda,x))
\text{ as $x\uparrow b$,} \lb{2.9a} \\
&\int_a^c dx \, p(x)^{-1}u_a(\lambda,x)^{-2}=\int_c^b dx \, 
p(x)^{-1}u_b(\lambda,x)^{-2}=\infty,  \lb{2.10} \\
&\int_a^c dx \, p(x)^{-1}{\hat u_a(\lambda,x)}^{-2}<\infty, \quad 
\int_c^b dx \, p(x)^{-1}{\hat u_b(\lambda,x)}^{-2}<\infty. \lb{2.11}
\end{align}
\end{theorem}
%%%%%%%%%%%%%%

Here
\begin{equation}
W(u,v)(x) = u(x)(pv')(x) - (pu')(x)v(x), \quad x\in (a,b),    \lb{2.11a}
\end{equation}
denotes the Wronskian of $u$ and $v$, assuming $u, (pu'), v, (pv') \in C((a,b))$. In 
particular, if $\ell u_j = z_j u_j$, $z_j \in \bbC$, then 
\begin{equation}
\f{d}{dx} W(u_1(z_1,x), u_2(z_2,x)) = (z_1 - z_2) r(x) u_1(z_1,x) u_2(z_2,x), \quad x \in (a,b).   \lb{2.11b} 
\end{equation}

%%%%%%%%%%%%%% 
\begin{definition} \lb{d2.6}
Assume Hypothesis \ref{h2.1} and let $\lambda\in\bbR$. Then
$u_a(\lambda,\cdot\,)$ $($resp., $u_b(\lambda,\cdot\,)$$)$ in Theorem
\ref{t2.5}\,$(iii)$ is called a {\it principal} $($or {\it minimal}\,$)$
solution of $\ell u=\lambda u$ at $a$ $($resp., $b$$)$. A solution 
$\ti u_a(\lambda,\cdot\,)$ $($resp., $\ti u_b(\lambda,\cdot\,)$$)$ of $\ell
u=\lambda u$ linearly independent of $u_a(\lambda,\cdot\,)$ $($resp.,
$u_b(\lambda,\cdot\,)$$)$ is called {\it nonprincipal} at $a$ $($resp., $b$$)$.
\end{definition}
%%%%%%%%%%%%%%

Principal and nonprincipal solutions are well-defined due to  
Lemma \ref{l2.7}\,$(i)$ below. 

%%%%%%%%%%%%%%%
\begin{lemma} [\cite{Ha82}] \lb{l2.7} Assume Hypothesis \ref{h2.1}. \\[1mm]
$(i)$ $u_a(\lambda,\cdot\,)$ and $u_b(\lambda,\cdot\,)$ in Theorem
\ref{t2.5}\,$(iii)$ are unique up to constant multiples. Moreover,
$u_a(\lambda,\cdot\,)$ and $u_b(\lambda,\cdot\,)$ are minimal solutions of
$\ell u=\lambda u$ in the sense that 
\begin{align}
u(\lambda,x)^{-1} u_a(\lambda,x)&=\oh(1) \text{ as $x\downarrow a$,} 
\lb{2.12} \\ 
u(\lambda,x)^{-1} u_b(\lambda,x)&=\oh(1) \text{ as $x\uparrow b$,} \lb{2.12a}
\end{align}
for any other solution $u(\lambda,\cdot\,)$ of $\ell u=\lambda u$
$($which is positive near $a$, resp., $b$$)$ with
$W(u_a(\lambda,\cdot\,),u(\lambda,\cdot\,))\neq 0$, respectively, 
$W(u_b(\lambda,\cdot\,),u(\lambda,\cdot\,))\neq 0$. \\[1mm]
$(ii)$ Let $u(\lambda,\cdot\,)$ be any positive solution of $\ell u=\lambda
u$ near $a$ $($resp., $b$$)$. Then for $c_1>a$ $($resp., $c_2<b$$)$
sufficiently close to $a$ $($resp., $b$$)$, 
\begin{align}
\hat u_a(\lambda,x)&=u(\lambda,x)\int_x^{c_1}dx' \, 
p(x')^{-1}u(\lambda,x')^{-2} \lb{2.13} \\
\bigg(\text{resp., }\hat u_b(\lambda,x)&=u(\lambda,x)\int^x_{c_2}dx' \, 
p(x')^{-1}u(\lambda,x')^{-2}\bigg) \lb{2.14} 
\end{align} 
is a nonprincipal solution of $\ell u=\lambda u$ at $a$ $($resp.,
$b$$)$. If $\hat u_a(\lambda,\cdot\,)$ $($resp., $\hat
u_b(\lambda,\cdot\,))$ is a nonprincipal solution of $\ell u=\lambda u$
at $a$ $($resp., $b$$)$ then 
\begin{align}
u_a(\lambda,x)&=\hat u_a(\lambda,x)\int_a^{x}dx' \, 
p(x')^{-1}{\hat u_a(\lambda,x')}^{-2} \lb{2.15} \\
\bigg(\text{resp., } u_b(\lambda,x)&=\hat u_b(\lambda,x)\int^b_{x}dx' \, 
p(x')^{-1}{\hat u_b(\lambda,x')}^{-2}\bigg) \lb{2.16} 
\end{align} 
is principal at $a$ $($resp., $b$$)$.
\end{lemma}
%%%%%%%%%%%%

The following two theorems describe a fundamental link between spectral
theory and non-oscillation results.

%%%%%%%%%%%%%% 
\begin{theorem} [\cite{Ha82}] \lb{t2.8} 
Assume Hypothesis \ref{h2.1} and let $\lambda_0\in\bbR$. Then the
following assertions are equivalent: \\[1mm]
$(i)$ $T_{min} \geq \lambda_0I$. \\[1mm]
$(ii)$ There exists a positive $($distributional\,$)$ solution $u>0$ of
$\ell v=\lambda_0 v$ on $(a,b)$.
\end{theorem}
%%%%%%%%%%%%%%

For the proof of Theorem \ref{t2.8} one notes that Theorems XI.6.1 and XI.6.2 and Corollary
XI.6.1 in Hartman's monograph \cite{Ha82} extend to our more general
hypotheses on $p,q,r$ without modifications. In particular,  
item $(ii)$ implies item $(i)$ by Jacobi's factorization identity 
\begin{align}
\begin{split} 
& -(pg')'+h^{-1}(ph')'g=-h^{-1}(ph^2(g/h)')', \\
& \quad 0<h, \, ph' \in AC_{loc}((a,b)), \; g  \in \dom(T_{min}).   \lb{2.18}
\end{split} 
\end{align}

%%%%%%%%%%%%%%
\begin{theorem} [Dunford--Schwartz \cite{DS88}, Theorem~XIII.7.40, 
\cite{EGNT13}, Section~11] \lb{t2.9}  ${}$ \\
Suppose Hypothesis \ref{h2.1}. Then the following assertions hold: \\[1mm] 
$(i)$ $T_{min}$ is not bounded from below if and only if for all
$\lambda\in\bbR$, every solution $u(\lambda,\cdot \,)$ of $\ell u = \lambda u$ has infinitely
many zeros on $(a,b)$. \\[1mm]
$(ii)$ If $T_{min}$ is bounded from below and
$\mu_0=\inf(\sigma_{ess}(T))$ for some self-adjoint extension $T$ of
$T_{min}$, then, for $\lambda>\mu_0$, every solution $u(\lambda,\cdot \,)$ of $\ell u=\lambda u$ 
has infinitely many zeros on $(a,b)$, while, for $\lambda<\mu_0$, no 
solution $u(\lambda,\cdot \,)$ of $\ell u = \lambda u$ has infinitely many zeros on $(a,b)$.
\end{theorem}
%%%%%%%%%%%%%

Thus, the existence of positive solutions on $(a,b)$ can be used to characterize 
$\inf(\sigma(T))$ while the existence of nonoscillatory solutions can be used to characterize
$\inf(\sigma_{ess}(T))$. Without going into further details at this point, we note that under 
appropriate assumptions on the coefficients, these characterizations extend to elliptic partial
differential operators. We also note that eigenvalue counts in essential spectral gaps in terms 
of (renormalized) oscillation theory in terms of zeros of Wronskians, rather than zeros of 
eigenfunctions, was established in \cite{GST96}. For additional work in this direction we 
refer to \cite{KT08}--\cite{KT09}.  

In order to set up the connection between principal and Weyl--Titchmarsh solutions, we next recall 
Weyl's definition of the limit point property of $\ell$ at the endpoint $a$ (resp., $b$).

%%%%%%%%%%%%%% 
\begin{definition} \lb{d2.10}
Assume Hypothesis \ref{h2.1} and let $z\in \bbC \backslash \bbR$. Then $\ell$ is said 
to be in the {\it limit point case $($l.p.c.$)$ at $a$ $($resp., $b$$)$} if for some $($and hence for 
all\,$)$ $c \in (a,b)$, there exists a unique solution $($up to constant 
multiples\,$)$ $\psi_-(z, \cdot \,)$ $($resp., $\psi_+(z, \cdot \, )$$)$ of 
$\ell u = z u$ such that 
\begin{equation}
\psi_- (z, \cdot \,) \in L^2((a,c);rdx) \, \text{ \big(resp., $\psi_+ (z, \cdot \,) \in L^2((c,b);rdx)$\big)}. 
\lb{2.21} 
\end{equation} 
\end{definition}
%%%%%%%%%%%%%% 

The constants permitted in Definition \ref{d2.10} (while of course $x$-independent) are 
generally $z$-dependent. 

One notes that $L^2$-solutions $u_{\pm}(z,\cdot \,)$ of $\ell u = z u$ in a neighborhood of 
$a$ and $b$ always exist. What singles out the limit point case for $\ell$ at $a$ or $b$ is the 
uniqueness (up to constant multiples) of the $L^2$-solution $\psi_-(z,\cdot \,)$, respectively, 
$\psi_+(z,\cdot \,)$ in Definition \ref{d2.10}. 

Any solution of $\ell u = z u$ satisfying the square 
integrability in \eqref{2.21} in a neighborhood of $a$ (resp., $b$), independent of whether it 
is unique up to constant multiples or not, is called a {\it Weyl--Titchmarsh solution} of 
$\ell u = z u$ near $a$ (resp., $b$). 

We continue with the fact that nonoscillatory behavior at one end point plus a simple condition 
on $r/p$ implies the limit point property at that endpoint: 

%%%%%%%%%%% 
\begin{lemma} [Hartman \cite{Ha48}, see also \cite{EGNT13}, Section~11, \cite{Ge93}, 
\cite{Re76}, \cite{Re51}] \lb{l2.11} ${}$ \\
Assume Hypothesis \ref{h2.1}, let $c \in (a,b)$, and suppose that for some $\lambda_0\in\bbR$,
$\ell-\lambda_0$ is nonoscillatory near $d \in \{a,b\}$. Then, if 
\begin{equation}
\bigg|\int_c^d dx\, [r(x)/p(x)]^{1/2}\bigg| =\infty,
\end{equation}
$\ell$ is in the limit point case at $d$. 
\end{lemma}
%%%%%%%%%%% 

Hartman's elegant proof of Lemma \ref{l2.11} in \cite{Ha48} is based on an application of 
(non)principal solutions of $\ell u = \lambda u$. 

In this context we also mention the following limit point result due to Kurss \cite{Ku67} (for 
the special case $r=1$): If for some  
$c \in (a,b)$, $q(x) \geq q_0(x)$ for a.e.~$x \in (a,c]$ (resp., $x \in [c,b)$), and the reference coefficient $q_0$ gives rise to a differential expression $\ell_0 = r^{-1}[-(d/dx) p (d/dx) + q_0]$ being in the limit point case and nonoscillatory at $a$ (resp., $b$), then $\ell = r^{-1}[-(d/dx) p (d/dx) + q]$ is in the limit point case at $a$ (resp., $b$). Kurss uses a Sturm comparison theorem to show that 
if for some $\lambda \in \bbR$, $\ell u = \lambda u$ has a solution $u_0 > 0$ on $(a,c]$ (resp., $[c,b)$) dominating a 
solution $v_0 > 0$ of $\ell_0 v_0 = \lambda v_0$ on $(a,c]$ (resp., $[c,b)$) such that 
$v_0 \notin L^2((a,c);dx)$ (resp., $v_0 \notin L^2((c,b);dx)$) 
in the sense that 
\begin{equation} 
u_0 \geq v_0 > 0 \, \text{ on $(a,c]$ (resp., $[c,b)$).} 
\end{equation} 
Thus, increasing the potential increases the non-$L^2$-solution, in particular, 
$u_0 \notin L^2((a,c);dx)$ (resp., $u_0 \notin L^2((c,b);dx)$). We are indebted to Hubert Kalf 
\cite{Ka15} for kindly pointing out to us the relevance of \cite{Ku67} and for a detailed discussion 
of the general case $r \neq 1$. 

%%%%%%%%%%%%%
\begin{remark} \lb{r2.12}
Assuming $\ell$ to be nonoscillatory near $a$ (resp., $b$) for some $\lambda_a \in \bbR$ 
(resp., $\lambda_b \in \bbR$), and in the limit point case at $a$ (resp., $b$), one 
recalls that $\psi_-$ (resp., $\psi_+$) in \eqref{2.21} analytically extends to $z < \lambda_a$ 
(resp., $z < \lambda_b$).   
In particular, for fixed $x \in (a,b)$, $\psi_- (\, \cdot \,, x)$ (resp., 
$\psi_+ (\, \cdot \,, x)$) is analytic in $\bbC \backslash [\lambda_a,\infty)$ (resp., 
$\bbC \backslash [\lambda_b,\infty)$). For more details in this context we refer to the comments following \cite[Proposition~1.1]{GST96}.
\end{remark} 
%%%%%%%%%%%%%

Next, we fix a reference point $x_0 \in (a,b)$, and introduce the normalized 
solutions $\phi(z, \cdot \, , x_0)$ and $\theta(z, \cdot \, , x_0)$ of $\ell u = z u$ 
by 
\begin{align}
\begin{split} 
& \phi(z, x_0 , x_0) = 0, \quad [p(x) \phi'(z, x , x_0)]_{x=x_0} = 1, \\
& \theta(z, x_0 , x_0) = 1, \quad [p(x) \theta'(z, x , x_0)]_{x=x_0} = 0,   \lb{2.22} 
\end{split}
\end{align}
with prime ${}'$ denoting $\partial/\partial x$, one infers (from the $z$-independence of the 
initial conditions in \eqref{2.22}) that for fixed $x \in (a,b)$, $\phi(\, \cdot \,, x,x_0)$ and 
$\theta(\, \cdot \,, x,x_0)$ are entire with respect to $z \in \bbC$ and that 
\begin{equation}
W (\theta(z, \cdot \, , x_0), \phi(z, \cdot \, , x_0)) = 1, \quad z \in \bbC, \; x_0 \in (a,b). 
\end{equation}  

Consequently, if $u_{\pm}(z, \cdot \,)$ denote {\it any} nontrivial square integrable solutions of 
$\ell u = z u$ in a neighborhood of $a$ and $b$, that is, for some 
(and hence for all) $c \in (a,b)$, 
\begin{equation}
u_+ (z, \cdot \,) \in L^2((c,b);rdx), \quad u_- (z, \cdot \,) \in L^2((a,c);rdx), 
\lb{2.26} 
\end{equation} 
one obtains $u_{\pm} (z, x_0) \neq 0$, and 
\begin{equation}
u_{\pm} (z, x) u_{\pm} (z, x_0)^{-1} = \theta(z, x, x_0) + \phi(z, x, x_0) m_{\pm}(z,x_0), \quad 
z \in \bbC\backslash\bbR, \; x, x_0 \in (a,b),    \lb{2.27}
\end{equation}
for some coefficients $m_{\pm}(\cdot \,,x_0)$, the Weyl--Titchmarsh functions associated with $\ell$.

The function $m_- (z,x_0)$ (resp., $m_+ (z,x_0)$) is uniquely determined  
if and only if $\ell$ is in the limit point case at $a$ (resp., $b$). In this case 
$u_- (z, \cdot \,)$ (resp., $u_+ (z, \cdot \,)$) coincides up to $z$-dependent constant multiples with 
$\psi_- (z, \cdot \,)$ (resp., $\psi_+ (z, \cdot \,)$) in \eqref{2.21}. 

Moreover, 
$\pm m_{\pm}(\, \cdot \, , x_0)$ are Nevanlinna--Herglotz functions, that is, 
for all $x_0 \in (a,b)$, 
\begin{equation}
m_{\pm}(\, \cdot \, , x_0) \, \text{ are analytic in $\bbC \backslash \bbR$,}
\end{equation} 
and 
\begin{equation}
\pm \Im(m_{\pm}(z,x_0)) > 0, \quad z \in \bbC_+. 
\end{equation} 
In addition, for all $x_0 \in (a,b)$, $m_{\pm}(\, \cdot \, , x_0)$ satisfy 
\begin{equation}
m_{\pm}(z,x_0) = \ol{m_{\pm} ({\ol z},x_0)}, \quad z \in \bbC_+. 
\end{equation}
Finally, one also infers for all $z \in \bbC \backslash \bbR$, $x_0 \in (a,b)$, 
\begin{align}
& W(u_+ (z, \cdot \,), u_- (z, \cdot \,)) = [m_-(z,x_0) - m_+(z,x_0)] u_+ (z, x_0) u_- (z, x_0),     \\ 
& m_{\pm}(z,x_0) = [p(x_0) u_{\pm}^{\prime} (z, x_0)] / u_{\pm} (z, x_0). 
\end{align}

Given these preparations we can finally state the main result of this section which identifies 
principal and Weyl--Titchmarsh solutions at an endpoint where $\ell$ is nonoscillatory and in 
the limit point case:

%%%%%%%%%%%%
\begin{theorem} \lb{t2.13} 
Assume Hypothesis \ref{h2.1}. \\
$(i)$ If $\ell$ is nonoscillatory and in the limit point case at $a$, then there exists 
$\lambda_a \in \bbR$, such that for all $\lambda < \lambda_a$, $x,x_0 \in (a,b)$, with $x, x_0$ 
to the left of the first zero of $\psi_- (\lambda, \cdot \,), u_a(\lambda,\cdot \,)$ $($if any\,$)$,   
\begin{equation}
\psi_-(\lambda,x) \psi_- (\lambda, x_0)^{-1} = u_a(\lambda,x) u_a(\lambda,x_0)^{-1}, \lb{2.35}
\end{equation}
that is, $\psi_-(\lambda,\cdot \,)$ and $u_a(\lambda, \cdot \,)$, $\lambda < \lambda_a$ 
are constant multiples of each other. \\[1mm]
$(ii)$ If $\ell$ is nonoscillatory and in the limit point case at $b$, then there exists 
$\lambda_b \in \bbR$, such that for all $\lambda < \lambda_b$, $x,x_0 \in (a,b)$, with $x, x_0$ 
to the right of the last zero of $\psi_+ (\lambda, \cdot \,), u_b(\lambda,\cdot \,)$ $($if any\,$)$,   
\begin{equation}
\psi_+(\lambda,x) \psi_+ (\lambda, x_0)^{-1} = u_b(\lambda,x) u_b(\lambda,x_0)^{-1},  \lb{2.34}
\end{equation}
that is, $\psi_+(\lambda,\cdot \,)$ and $u_b(\lambda, \cdot \,)$, $\lambda < \lambda_b$, 
are constant multiples of each other. 
\end{theorem}
%%%%%%%%%%%%
\begin{proof}
It suffices to consider the case of $\psi_+$ and $u_b$. Then, if $\psi_+$ is a nonprincipal 
solution of $\ell u = \lambda u$, 
Lemma \ref{l2.7}\,$(i)$ implies the existence of $C_+ > 0$ and $c \in (a,b)$, such that for all 
$\lambda < \lambda_b$ and for all $x \in (c,b)$,  
\begin{equation}
|u_b(\lambda, x)| \leq C_+ |\psi_+(\lambda, x)|. 
\end{equation} 
Thus, $u_b(\lambda, \cdot \,) \in L^2((c,b); r dx)$. But since by hypothesis $\psi_+$ and $u_b$ 
are linearly independent, $W(\psi_+(\lambda, \cdot \,), u_b(\lambda, \cdot \,)) \neq 0$, this 
contradicts the limit point hypothesis at $b$ which yields precisely one 
$L^2((c,b); r dx)$-solution up to constant (generally, $\lambda$-dependent) multiples. 
\end{proof}
%%%%%%%%%%%%

In particular, if $T_{min}$ is bounded from below by $\lambda_0 \in \bbR$ and essentially self-adjoint, 
then for all $\lambda < \lambda_0$, principal and Weyl--Titchmarsh solutions at an endpoint 
coincide up to constant ($\lambda$-dependent) multiples.   
 
We briefly follow up with the connection between Green's functions and principal solutions for 
Sturm--Liouville operators, illustrating once more the relevance of principal solutions.

%%%%%%%%%%%%%
\begin{lemma} \lb{l2.14}
Assume Hypothesis \ref{h2.1} and suppose that $T_{min} \geq \lambda_0 I$ for some 
$\lambda_0 \in \bbR$. In addition, asume that $\ell$ is in the limit point case at $a$ and $b$. 
Then 
\begin{equation}
\ol{T_{min}} = T_{max} := T  
\end{equation} 
is the unique self-adjoint extension of $T_{min}$ in $L^2((a,b);rdx)$ and for any 
$x_0 \in (a,b)$, 
\begin{align}
\begin{split} 
0 < G (\lambda,x,x')&=\bigg(\int_{x_0}^b dt \, p(t)^{-1}
u_a(\lambda,t)^{-2}\bigg)u_b(\lambda,x_0)^{-1}u_a(\lambda,x_0)\times 
\\
&\quad \times \begin{cases} u_a(\lambda,x)u_b(\lambda,x'), &
a < x\leq x' < b, \\ u_a(\lambda,x')u_b(\lambda,x), &
a < x' \leq x < b, \end{cases} \quad  \lambda < \lambda_0,   \lb{2.28}
\end{split} 
\end{align}
is the positive Green's function of $T$. Here we abbreviated   
\begin{equation}
G(z,x,x') = (T - z I)^{-1}(x,x'), \quad  x,x' \in (a,b), \; z \in \bbC \backslash [\lambda_0, \infty).   \lb{2.29}
\end{equation}
\end{lemma}
%%%%%%%%%%%%%

As a consequence of Theorem \ref{t2.13}, $u_a(\lambda,\cdot \,)$ and $u_b(\lambda,\cdot \,)$ 
in the Green's function representation \eqref{2.28} 
can be replaced by $\psi_-(z,\cdot \,)$ and $\psi_+(z,\cdot \,)$. More precisely, an additional 
analytic continuation with respect to $z \in \bbC \backslash [\lambda_0, \infty)$ yields 
\begin{align}
& G (z,x,x') = \f{1}{W(\psi_+(z,\cdot \,), \psi_-(z,\cdot \,))} \begin{cases} \psi_-(z,x) \psi_+(z,x'), &
a < x\leq x' < b, \\ 
\psi_+(z,x) \psi_-(z,x'), & a < x' \leq x < b, \end{cases}    \no \\[1mm] 
& \hspace*{8.4cm} z \in \bbC \backslash [\lambda_0, \infty),   
\end{align}
where for all $z \in \bbC \backslash [\lambda_0, \infty)$, $x_0 \in (a,b)$, 
\begin{align}
& W(\psi_+ (z, \cdot \,), \psi_- (z, \cdot \,)) = [m_-(z,x_0) - m_+(z,x_0)] \psi_+ (z, x_0) \psi_- (z, x_0),     \\ 
& m_{\pm}(z,x_0) = [p(x_0) \psi_{\pm}^{\prime} (z, x_0)] / \psi_{\pm} (z, x_0). 
\end{align} 

The material in Lemma \ref{l2.3}--Theorem \ref{t2.9}, Lemma \ref{l2.11}, and Lemma \ref{l2.14} 
(and considerably more) is discussed in great detail in \cite{GZ91} (with special emphasis on 
the Friedrichs extension $T_F$ of $T_{min}$), and under more general conditions on $\ell$ and 
its coefficients in \cite{EGNT13}. 

We conclude this section by recalling a known formula for $m_- (\, \cdot \, , x_0)$ (resp., 
$m_+ (\, \cdot \, , x_0)$) whenever $\ell$ is in the limit point case and nonoscillatory at 
$a$ (resp., $b$): Assuming the limit point case of $\ell$ at $a$ (resp., $b$), it is well-known that 
\begin{align}
\begin{split} 
& m_-(z, x_0) = - \lim_{x \downarrow a} \theta(z,x,x_0)/\phi(z,x,x_0), 
\quad z \in \bbC\backslash\bbR,    \lb{2.42} \\
& (\text{resp., } \, m_+(z, x_0) = - \lim_{x \uparrow b} \theta(z,x,x_0)/\phi(z,x,x_0), \quad z \in \bbC\backslash\bbR).
\end{split}
\end{align}
Next, fix $z \in \bbC$ and suppose that $v(z, \cdot \,, x_0)$ satisfies $\ell u = z u$ and 
$v(z,x) \neq 0$ for $x \in [x_0,b)$, then clearly $w(z, \cdot \,, x_0)$ defined by 
\begin{equation}
w(z,x) = v(z,x,x_0) \bigg[C_1 + C_2 \int_{x_0}^x dx' \, p(x')^{-1} v(z,x',x_0)^{-2}\bigg], \quad 
x \in [x_0,b),    \lb{2.43} 
\end{equation}
is a solution of $\ell u = z u$ satisfying $W(v(z, \cdot \, ,x_0), w(z, \cdot \, ,x_0)) = C_2$. 
An elementary application of these facts to $\phi(z, \cdot \, , x_0)$ and $\theta(z, \cdot \, ,x_0)$, 
taking into account that $\theta(z, x, x_0) \neq 0$, $z \in \bbC\backslash\bbR$, $x, x_0 \in (a,b)$, yields
\begin{equation}
\phi(z,x, x_0) = \theta(z,x,x_0) \int_{x_0}^x dx' \, p(x')^{-1} \theta(z,x',x_0)^{-2}, \quad 
z\in \bbC\backslash\bbR, \; x, x_0 \in (a,b).    \lb{2.44}   
\end{equation}
Insertion of \eqref{2.44} into \eqref{2.42} yields the interesting formula,
\begin{equation}
m_+(z, x_0) = - \lim_{x \uparrow b} \bigg[\int_{x_0}^x dx' \, p(x')^{-1} \theta(z,x',x_0)^{-2}\bigg]^{-1}, 
\quad z \in \bbC\backslash\bbR.     \lb{2.45} 
\end{equation}
If in addition, $\ell$ is nonoscillatory at $b$, analytic continuation of both sides in \eqref{2.45} 
with respect to $z$ permits one to extend \eqref{2.45} to all $z \in \bbC \backslash [\lambda_b, \infty)$, 
with $\lambda_b$ as in Theorem \ref{t2.13}\,$(ii)$. We also note that for $\lambda < \lambda_b$, 
the expression 
\begin{equation}
\bigg[\int_{x_0}^x dx' \, p(x')^{-1} \theta(\lambda,x',x_0)^{-2}\bigg]^{-1}, \quad \lambda < \lambda_b, 
\lb{2.46} 
\end{equation}
is strictly monotonically decreasing with respect to $x$ and hence the existence of the limit of 
the integral in \eqref{2.46} as $x \uparrow b$ is guaranteed and one obtains
\begin{equation}
m_+(\lambda, x_0) = - \bigg[\int_{x_0}^b dx' \, p(x')^{-1} \theta(\lambda,x',x_0)^{-2}\bigg]^{-1}, \quad 
\lambda < \lambda_b,     \lb{2.47} 
\end{equation}
with $\int_{x_0}^b dx' \, \cdots$ in \eqref{2.47} representing a Lebesgue integral.  

We first found \eqref{2.47} mentioned without proof in a paper by 
Kotani \cite{Ko85}. Kotani kindly alerted us to a paper by Kac and Krein \cite{KK74}, where 
such a formula is discussed near the end of their section 2, but the precise history of \eqref{2.47} 
is unknown to us at this point. We will provide a detailed derivation of \eqref{2.45}, \eqref{2.47} in 
the matrix-valued context in Section \ref{s3}. 

Next, replacing $\phi (z, \cdot \, ,x_0), \theta (z, \cdot \, , x_0)$ satisfying $\ell u = z u$ and 
\eqref{2.22} by the more general $\phi_{\alpha} (z, \cdot \, ,x_0), \theta_{\alpha} (z, \cdot \, , x_0)$ 
satisfying $\ell u = z u$ and 
\begin{align}
\begin{split} 
& \phi_{\alpha}(z, x_0 , x_0) = - \sin(\alpha), 
\quad [p(x) \phi_{\alpha}'(z, x , x_0)]|_{x=x_0} = \cos(\alpha), \\
& \theta_{\alpha}(z, x_0 , x_0) = \cos(\alpha), 
\quad [p(x) \theta_{\alpha}'(z, x_0 , x_0)]|_{x=x_0} = \sin(\alpha),   \lb{2.48} 
\end{split}
\end{align}
for some $\alpha \in [0, \pi)$, and hence replacing \eqref{2.27} by
\begin{align}
\begin{split} 
u_{\pm,\alpha} (z, x, x_0) = \theta_{\alpha} (z, x, x_0) 
+ \phi_{\alpha}(z, x, x_0) m_{\pm,\alpha}(z,x_0),&      \\
z \in \bbC\backslash\bbR, \; x, x_0 \in (a,b),&      \lb{2.49}
\end{split} 
\end{align}
for appropriate Weyl--Titchmarsh coefficients 
$m_{\pm, \alpha}(\, \cdot \,,x_0)$, one obtains along the lines leading to \eqref{2.44} for $z \in \bbC \backslash \bbR$, 
\begin{align}
\phi_{\alpha}(z,x,x_0) &= - \tan(\alpha) \theta_{\alpha}(z,x,x_0) +  \theta_{\alpha}(z,x,x_0)
\int_{x_0}^x \f{dx'}{p(x') \theta_{\alpha}(z,x',x_0)^2},    \\
& \hspace*{6.55cm} \alpha \in [0, \pi) \backslash \{\pi/2\},    \no \\
\theta_{\alpha}(z,x,x_0) &= - \cot(\alpha) \phi_{\alpha} (z,x,x_0) 
- \phi_{\alpha} (z,x,x_0) 
\int_{x_0}^x \f{dx'}{p(x') \phi_{\alpha}(z,x',x_0)^2},    \\
& \hspace*{7.3cm} \quad \alpha \in (0, \pi),   \no 
\end{align}
and using 
\begin{equation}
m_{+,\alpha}(z, x_0) = - \lim_{x \uparrow b} \theta_{\alpha}(z,x,x_0)/\phi_{\alpha}(z,x,x_0), \quad z \in \bbC\backslash\bbR, 
\end{equation}
one now obtains 
\begin{align}
& m_{+,\alpha}(z) = \begin{cases}
\bigg[\tan(\alpha) - \lim_{x \uparrow b}\int_{x_0}^x 
dx' \, p(x')^{-1} \theta_{\alpha}(z,x',x_0)^{-2}\bigg]^{-1}, \hspace*{-1mm} 
& \alpha \in [0,\pi) \backslash \{\pi/2\},   \\[2mm] 
\cot(\alpha) + \lim_{x \uparrow b} \int_{x_0}^x dx' \, p(x')^{-1} 
\phi_{\alpha}(z,x',x_0)^{-2}, & \alpha \in (0,\pi), 
\end{cases}     \no \\
& \hspace*{9.6cm} z \in \bbC \backslash \bbR,      \lb{2.50} 
\end{align} 
whenever $\ell$ is in the limit point case at $b$ (and similarly for $z < \lambda_{b,\alpha}$ for an appropriate $\lambda_{b,\alpha} \in \bbR$, and a Lebesgue integral $\int_{x_0}^b dx' \cdots$ if $\ell$ is also nonoscillatory at $b$). 

Replacing $\lim_{x \uparrow b}\int_{x_0}^x dx' \, \cdots$ by 
$- \lim_{x \downarrow a} \int_x^{x_0} dx' \, \cdots$, all formulas for 
$m_+(\, \cdot \, ,x_0)$ and $m_{+, \alpha}(\, \cdot \, ,x_0)$ immediately extend to  
$m_-(\, \cdot \, ,x_0)$ and $m_{-, \alpha}(\, \cdot \, ,x_0)$, assuming $\ell$ to be in the limit point case at $a$ (and analogously if $\ell$ is also nonoscillatory at 
$a$).

%%%%%%%%%%%%%%%%%%%%%%%%%%%%%%
%%%%%%%%%%%%%%%%%%%%%%%%%%%%%%
\section{Matrix-Valued Principal Solutions}  \lb{s3} 
%%%%%%%%%%%%%%%%%%%%%%%%%%%%%%
%%%%%%%%%%%%%%%%%%%%%%%%%%%%%%

This section is devoted to an extension of some of the basic results on principal 
solutions of the previous Section \ref{s2} to those associated with matrix-valued 
singular Sturm--Liouville operators. 

Matrix oscillation theory relevant to this paper originated with Hartman \cite{Ha57} and 
Reid \cite{Re58}. The literature on oscillation theory for systems of differential equations 
is so rich by now that we cannot possibly offer a comprehensive list of references. Hence, 
we restrict ourselves primarily to a number of monographs by Coppel \cite[Ch.\ 2]{Co71}, 
Hartman \cite[Sects.\ X.10, X.11]{Ha82}, Hille \cite[Sect.\ 9.6]{Hi69}, \cite{Hi76}, 
Kratz \cite[Chs.\ 4, 7]{Kr95}, Reid \cite[Ch.\ VII]{Re71}, \cite[Ch.\ V]{Re80}, 
Rofe-Beketov and Kholkin \cite[Chs.\ 1--4]{RK05},
and a few additional such as \cite{AL78a}, \cite{AS81}, \cite{BEM87}, \cite{Et69}, 
\cite{FJN03}, \cite{FJN03a}, \cite{FJN06}, \cite{FJNN11}, \cite{GJ89}, \cite{Ja64}, 
\cite{Re64}, \cite{RH77}, and \cite{Ya92}. 

The basic Weyl--Titchmarsh theory and general spectral theory for matrix-valued singular 
Sturm--Liouville operators as well as the more general case of singular Hamiltonian systems 
has been derived in detail by Hinton and Shaw \cite{HS81}--\cite{HS86} 
(we also refer to \cite[Ch.\ 10]{At64}, \cite{CG01}, \cite{CG02}, \cite{CG03}, \cite{CGHL00}, 
\cite{FJN03a}, \cite{GKM00}, \cite{GT00}, \cite[Sect.\ 10.7]{Hi69}, \cite{HS93}, 
\cite{HS97}, \cite{HS98}, \cite{Jo87}, \cite{JNO00}, \cite{KS88}, \cite{Kr89a}, \cite{Kr89b}, \cite{KR74}, 
\cite{LM03}, \cite{Or76}, \cite{Re96}, \cite{Sa92}, \cite{We87} for pertinent spectral results in this connection). 

In the following we take these developments for granted and only focus on the required changes 
in Section \ref{s2} in connection with principal solutions which are implied by inherent noncommutativity issues due to the matrix-valued setting.  

The basic assumptions for this section then read as follows:

%%%%%%%%%%
\begin{hypothesis} \lb{h3.1}
Let $-\infty\leq a<b\leq\infty$ and suppose that $P,Q,R \in \bbC^{m \times m}$, $m \in \bbN$, have 
$($Lebesgue$\,)$ measurable entries on $(a,b)$, and that 
\begin{align}
\begin{split} 
& P>0, R>0 \text{ a.e.\ on $(a,b)$, $Q=Q^*$ is self-adjoint},     \\
& P^{-1}, Q, R \in L^1_{loc} ((a,b); dx)^{m \times m}.      \lb{3.1}
\end{split} 
\end{align}
\end{hypothesis}
%%%%%%%%%%%%%

In addition, we introduce the Hilbert space of $\bbC^m$-valued elements,
\begin{align}
\begin{split}
& L^2((a,b); R dx; \bbC^m) = \bigg\{U = (U_1,\dots,U_m)^\top, \text{$U_k$ (Lebesgue) measurable,}  
 \\
& \hspace*{3.8cm} 1\leq k \leq m  \, \bigg| \, 
 \int_{(a,b)} dx \, (U(x), R(x) U(x))_{\bbC^m} < \infty\bigg\},      \lb{3.2} 
 \end{split} 
\end{align}
with associated scalar product
\begin{align}
\begin{split} 
& (U,V)_{L^2((a,b); R dx; \bbC^m)} = \int_{(a,b)} dx \, (U(x), R(x) V(x))_{\bbC^m}, \\
& \hspace*{4.05cm} U, V \in L^2((a,b); R dx; \bbC^m).   
\end{split} 
\end{align} 

Here ${(\dots)}^{\top}$ indicates a column vector in $\bbC^m$ and $(\, \cdot \, , \cdot \,)_{\bbC^m}$ represents the standard scalar product in $\bbC^m$, that is,
\begin{equation}
(w_1,w_2)_{\bbC^m} = \sum_{k=1}^m \ol{w_{1,k}} w_{2,k}, \quad 
w_j = (w_{j,1},\dots,w_{j,m})^{\top} \in \bbC^m, \; j = 1,2. 
\end{equation}

Given Hypothesis \ref{h3.1}, we consider the differential expression
\begin{equation}
L = R^{-1}\bigg(-\f{d}{dx}P\f{d}{dx} + Q\bigg), \quad 
-\infty\leq a<x<b\leq\infty,   \lb{3.4}
\end{equation}
and once more define the minimal operator $T_{min}$ and maximal operator 
$T_{max}$ in $L^2((a,b);R dx; \bbC^m)$ associated with $L$ by 
\begin{align}
& T_{min}u=L u,   \no \\
& u\in \dom (T_{min}) = \big\{v \in L^2((a,b);R dx; \bbC^m)\,\big|\, v, 
Pv' \in AC_{loc} ((a,b))^{m \times 1};     \lb{3.5} \\
& \hspace*{3cm}  \supp\,(v)\subset (a,b) \, \text{compact}; \, L v \in L^2((a,b);Rdx; \bbC^m)\big\},  \no \\
& T_{max}u=L u,    \no \\ 
& u\in \dom (T_{max})= \big\{v \in L^2((a,b);R dx; \bbC^m)\,\big|\, v, 
Pv'\in AC_{loc} ((a,b))^{m \times 1};    \lb{3.6}  \\ 
& \hspace*{3.05cm} L v \in L^2((a,b);R dx; \bbC^m)\big\},  \no 
\end{align}
respectively. Here $AC_{loc}((a,b))^{m \times n}$ denotes the set of $m \times n$ matrices,  
$m,n \in \bbN$, with locally absolutely continuous entries on $(a,b)$ (we will use the analogous 
in connection with $C((a,b))^{m \times n}$ below).  

Again, $T_{min}$ is densely defined and 
\begin{equation}
{T_{min}}^*=T_{max}, \quad T_{max}^* = \ol{T_{min}}. \lb{3.7}
\end{equation}

In the following, matrix-valued solutions $U(z,\cdot\,)$ of $L U = z U$ for some $z\in\bbC$, 
are always assumed to be distributional solutions, in addition, we either assume the vector-valued 
\begin{equation}
u(z,\cdot\,), \, Pu'(z,\cdot\,)\in AC_{loc} ((a,b))^{m \times 1},     \lb{3.8}
\end{equation}
or the $m \times m$ matrix-valued case 
\begin{equation}
U(z,\cdot\,), \, PU'(z,\cdot\,)\in AC_{loc} ((a,b))^{m \times m},    \lb{3.9}
\end{equation}
in this context. In fact, assuming 
$U, (PU)', V, (PV)' \in C((a,b))^{m \times m}$, one introduces the matrix-valued Wronskian of 
$u$ and $v$ by 
\begin{equation}
W(U,V)(x) = U(x)(PV)'(x) - (PU)'(x)V(x), \quad x\in (a,b),    \lb{3.10}
\end{equation}
and if $U_j$ are $m \times m$ matrix solutions of $L U_j = z_j U_j$, $z_j \in \bbC$, then 
\begin{equation}
\f{d}{dx}W(U_1(\ol{z_1},x)^*, U_2(z_2,x)) = (z_1 - z_2) U_1(\ol{z_1},x)^* R(x) U_2(z_2,x), 
\quad x \in (a,b).   \lb{3.11} 
\end{equation}

%%%%%%%%%%%%%% 
\begin{definition} \lb{d3.2}
Assume Hypothesis \ref{h3.1} and let $z\in \bbC \backslash \bbR$. Then $L$ is said 
to be in the {\it limit point case $($l.p.c.$)$ at $a$ $($resp., $b$$)$} if for some $($and hence for 
all\,$)$ $c \in (a,b)$, there exists a unique invertible $m \times m$ matrix-valued solution 
$($up to constant multiples by right multiplication with invertible $m \times m$ matrices\,$)$ 
$\Psi_- (z, \cdot \,)$ $($resp., $\Psi_+ (z, \cdot \, )$$)$ of $L U = z U$ such that the $m \times m$ 
matrices 
\begin{equation}
\int_a^c dx \,  \Psi_- (z,x)^* R(x) \Psi_- (z,x)    
\, \text{ \bigg(resp., $\int_c^b dx \,  \Psi_+ (z,x)^* R(x) \Psi_+ (z,x)$\bigg)}
\lb{3.12} 
\end{equation} 
exist. 
\end{definition}
%%%%%%%%%%%%%% 

Again, the constant invertible $m \times m$ matrices permitted in connection with right multiplication in Definition \ref{d3.2} (while of course $x$-independent) are generally $z$-dependent. 

Given the analogy to the scalar case $m=1$, any solution of $L U = z U$ satisfying the square 
integrability condition \eqref{3.12} in a neighborhood of $a$ (resp., $b$), independent of 
uniqueness up to right multiplication by constant invertible matrices, will be called a 
($m \times m$ matrix-valued) Weyl--Titchmarsh solution of $L U = z U$ near $a$ (resp., $b$). 

%%%%%%%%%%%%%
\begin{remark}\lb{r3.3}
Assuming there exists a $\lambda_a \in \bbR$ (resp., $\lambda_b\in \bbR$) for which $(u,[T_{min} - \lambda_a I] u)_{L^2((a,b); R dx; \bbC^m)}\geq 0$ for all $u \in \dom(T_{min})$ with $u=0$ in a neigborhood of $b$ (resp., $(u,[T_{min} - \lambda_b I] u)_{L^2((a,b); R dx; \bbC^m)}\geq 0$ for all $u \in \dom(T_{min})$ with $u=0$ in a neigborhood of $a$) and that $L$ is in the limit point case at $a$ (resp., $b$), then 
$\Psi_-$ (resp., $\Psi_+$) in \eqref{3.12} analytically extends to $z < \lambda_a$ (resp., $z < \lambda_b$). In particular, for fixed $x \in (a,b)$, $\Psi_- (\, \cdot \,, x)$ (resp., $\Psi_+ (\, \cdot \,, x)$) is analytic in 
$\bbC \backslash [\lambda_a,\infty)$ (resp., $\bbC \backslash [\lambda_b,\infty)$) (cf.\ the analogous 
Remark \ref{r2.12} in the scalar context). 
\end{remark} 
%%%%%%%%%%%%%

Next, we turn to a brief summary of the principal facts of Weyl--Titchmarsh theory in the present 
matrix-valued context. Again, we fix a reference point $x_0 \in (a,b)$, and introduce the normalized 
$m \times m$ matrix-valued solutions $\Phi(z, \cdot \, , x_0)$ and $\Theta(z, \cdot \, , x_0)$ 
of $L U = z U$ by 
\begin{align}
\begin{split} 
& \Phi(z, x_0 , x_0) = 0, \quad [P(x) \Phi'(z, x , x_0)]|_{x=x_0} = I_m, \\
& \Theta(z, x_0 , x_0) = I_m, \quad [P(x) \Theta'(z, x , x_0)]|_{x=x_0} = 0,   \lb{3.13} 
\end{split}
\end{align}
and note again that for fixed $x \in (a,b)$, $\Phi(\, \cdot \,, x,x_0)$ and 
$\Theta(\, \cdot \,, x,x_0)$ are entire with respect to $z \in \bbC$. Moreover, one verifies 
(cf., e.g., \cite[Sect.\ 2]{CG02}, \cite[Sect.\ 2]{GWZ13}) that for any $z \in \bbC$, $x_0 \in (a,b)$, 
\begin{align}
W (\Theta({\ol z}, \cdot \, , x_0)^*, \Phi(z, \cdot \, , x_0)) = I_m,    \lb{3.14} \\
W (\Phi({\ol z}, \cdot \, , x_0)^*, \Theta(z, \cdot \, , x_0)) = I_m,   \lb{3.14a} \\
W (\Phi({\ol z}, \cdot \, , x_0)^*, \Phi(z, \cdot \, , x_0)) = 0,       \lb{3.14b} \\
W (\Theta({\ol z}, \cdot \, , x_0)^*, \Theta(z, \cdot \, , x_0)) = 0,    \lb{3.14c}
\end{align}  
as well as, 
\begin{align}
& \Phi(z,x, x_0) \Theta({\ol z},x,x_0)^* - \Theta(z,x,x_0) \Phi({\ol z},x, x_0)^* = 0,    \lb{3.14d} \\
& [P(x) \Phi'(z,x, x_0)] [P(x) \Theta'({\ol z},x,x_0)]^* - [P(x) \Theta'(z,x,x_0)] [P(x) \Phi'({\ol z},x, x_0)]^* 
= 0,      \lb{3.14e} \\
& [P(x) \Phi'(z,x, x_0)]  \Theta({\ol z},x,x_0)^* - [P(x) \Theta'(z,x,x_0)] \Phi({\ol z},x, x_0)^* = I_m,  
\lb{3.14f} \\
&  \Theta(z,x,x_0) [P(x) \Phi'({\ol z},x, x_0)]^* - \Phi(z,x, x_0) [P(x) \Theta'({\ol z},x,x_0)]^* = I_m.   
\lb{3.14g} \\
\end{align}

Consequently, if $U_{\pm}(z, \cdot \,)$ denote {\it any} invertible square integrable $m \times m$ 
matrix-valued solutions of $L U = z U$ in a neighborhood of $a$ and $b$ in the sense that for some 
(and hence for all) $c \in (a,b)$, the $m \times m$ matrices 
\begin{equation}
\int_a^c dx \,  U_- (z,x)^* R(x) U_- (z,x), \quad \int_c^b dx \,  U_+ (z,x)^* R(x) U_+ (z,x), 
\lb{3.15} 
\end{equation} 
exist, one obtains 
\begin{equation}
U_{\pm} (z, x) U_{\pm} (z, x_0)^{-1} = \Theta(z, x, x_0) + \Phi(z, x, x_0) M_{\pm}(z,x_0), \quad 
z \in \bbC\backslash\bbR, \; x, x_0 \in (a,b),   \lb{3.16}
\end{equation}
for some $m \times m$ matrix-valued coefficients $M_{\pm}(z,x_0)\in \bbC^{m \times m}$, the 
Weyl--Titchmarsh matrices associated with $L$.

Again, the matrix $M_- (z,x_0)$ (resp., $M_+ (z,x_0)$) is uniquely determined if and only if $L$ 
is in the limit 
point case at $a$ (resp., $b$). In this case $U_- (z, \cdot \,)$ (resp., $U_+ (z, \cdot \,)$) 
coincides up to right multiplication by $z$-dependent constant matrices with $\Psi_- (z, \cdot \,)$ 
(resp., $\Psi_+ (z,x_0)$) in \eqref{3.12}. 

Moreover, 
$\pm M_{\pm}(\, \cdot \, , x_0)$ are $m \times m$ Nevanlinna--Herglotz matrices, that is, 
for all $x_0 \in (a,b)$, 
\begin{equation}
M_{\pm}(\, \cdot \, , x_0) \, \text{ are analytic in $\bbC \backslash \bbR$,}   \quad 
{\rm rank} \, (M_{\pm}(z , x_0)) =m, \; z \in \bbC_+,     \lb{3.17}
\end{equation} 
and 
\begin{equation}
\pm \Im(M_{\pm}(z,x_0)) > 0, \quad z \in \bbC_+.     \lb{3.18} 
\end{equation} 
(Here, in obvious notation, $\Im (M) = (2i)^{-1} (M - M^*)$, $M \in \bbC^{m \times m}$.) 
In addition, for all $x_0 \in (a,b)$, $M_{\pm}(\, \cdot \, , x_0)$ satisfy  
\begin{equation}
M_{\pm}(z,x_0) = M_{\pm} ({\ol z},x_0)^*, \quad z \in \bbC_+.      \lb{3.19} 
\end{equation}
Finally, one also infers for all $z \in \bbC \backslash \bbR$, $x_0 \in (a,b)$, 
\begin{align}
& W(U_+ ({\ol z}, \cdot \,)^*, U_- (z, \cdot \,)) 
= U_+ ({\ol z}, x_0)^* [M_-(z,x_0) - M_+(z,x_0)] U_- (z, x_0),    \lb{3.20} \\ 
& M_{\pm}(z,x_0) = P(x_0) U_{\pm}^{\prime} (z, x_0) U_{\pm} (z, x_0)^{-1}.   \lb{3.21}
\end{align}

Unraveling the crucial identities \eqref{3.19} and \eqref{3.21} results in the fundamental fact 
\begin{equation}
U_{\pm}({\ol z}, x)^* [P(x) U_{\pm}'(z,x)] = [P(x) U_{\pm}'({\ol z}, x)]^* U_{\pm}(z,x), \quad 
z \in \bbC \backslash \bbR,     \lb{3.22}
\end{equation} 
for $x \in (a,b)$. In particular,
\begin{equation}
W(U_{\pm} ({\ol z}, \cdot \,)^*, U_{\pm} (z, \cdot \,)) = 0.    \lb{3.23} 
\end{equation}

In other words, invertible square integrable $m \times m$ matrix-valued 
solutions of $L U = z U$ in the sense of \eqref{3.15} closely resemble {\it prepared} solutions in 
the sense of Hartman \cite{Ha57}. We use the term ``closely resemble'' as Hartman avoids the use 
of a complex spectral parameter and focuses on $z=0$ instead. The term ``prepared'' did not stick 
as one finds also the notions of {\it conjoined, isotropic}, and {\it self-conjugate} 
solutions in the literature in connection with the property \eqref{3.22} (resp., \eqref{3.23}). Be 
that as it may, isolating property \eqref{3.22} was definitely a crucial step in the spectral analysis 
of systems of differential equations as the following observations will demonstrate. 

%%%%%%%%%%
\begin{definition} \lb{d3.3}
Assume Hypothesis \ref{h3.1}, let $\lambda \in \bbR$, and suppose that $U(\lambda, \cdot \,)$ 
is an $m \times m$ matrix-valued solution of $L U = \lambda U$. Then $U(\lambda, \cdot \,)$ is 
called {\it self-conjugate} if 
\begin{equation}
W(U(\lambda, \cdot \,)^*, U(\lambda, \cdot \,)) = 0,     \lb{3.24}  
\end{equation}
equivalently, if 
\begin{equation}
U(\lambda, x)^* [P(x) U'(\lambda,x)] = [P(x) U'(\lambda, x)]^* U(\lambda,x), \quad 
x \in (a,b).      \lb{3.25}
\end{equation} 
\end{definition} 
%%%%%%%%%%

That is, $U(\lambda, x)^* [P(x) U'(\lambda,x)]$ is self-adjoint in \eqref{3.25} for all $x \in (a,b)$,  
and this is why we thought it most natural to follow those who adopted the term ``self-conjugate'' in 
connection with Definition \ref{d3.3}. While we 
could have extended Definition \ref{d3.3} immediately to $\lambda \in \bbC$ along the lines of 
\eqref{3.22}, \eqref{3.23}, we are eventually aiming at principal matrix-valued solutions which are 
typically considered for $\lambda \in \bbR$. 

Next, let $V_- (z,\cdot \,, c)$ (resp., $V_-+ (z,\cdot \,, c)$) be $m \times m$ matrix-valued solutions of 
$L U = z U$, invertible on the interval $(a,c]$ (resp., $[c,b)$) for some $c \in (a,b)$, satisfying 
property \eqref{3.22} on $(a,c]$ (resp., $[c,b)$). In particular,  
\begin{equation}
W(V_{\pm} ({\ol z}, \cdot \, ,c)^*, V_{\pm} (z, \cdot \, ,c)) = 0.   \lb{3.26} 
\end{equation}
We introduce 
\begin{align}
\begin{split} 
& W_-(z,x,c) = V_-(z,x,c) \bigg[C_{-,1}   \\
& \quad - \int_x^c dx' V_+(z,x',c)^{-1} P(x')^{-1} \big[V_+({\ol z}, x',c)^{-1}\big]^* C_{-,2}\bigg],    
\quad x \in (a,c],     \lb{3.28} 
\end{split} 
\end{align} 
and 
\begin{align}
\begin{split}
& W_+(z,x,c) = V_+(z,x,c) \bigg[C_{+,1}   \\
& \quad + \int_c^x dx' V_+(z,x',c)^{-1} P(x')^{-1} \big[V_+({\ol z}, x',c)^{-1}\big]^* C_{+,2}\bigg],     
\quad x \in [c, b),     \lb{3.27} 
\end{split} 
\end{align} 
where $C_{\pm,j} \in \bbC^{m \times m}$, $j=1,2$. Then straightforward computations yield
\begin{align}
& L W_- (z, \cdot \, , c) = z W_- (z, \cdot \, , c) \, \text{ on } \, (a,c],    \lb{3.30} \\
& L W_+ (z, \cdot \, , c) = z W_+ (z, \cdot \, , c) \, \text{ on } \, [c,b),    \lb{3.29} \\
& W(V_{\pm} ({\ol z}, \cdot \, ,c)^*, W_{\pm} (z, \cdot \, ,c)) = C_{\pm, 2},    \lb{3.31} \\
& W(W_{\pm} ({\ol z}, \cdot \, ,c)^*, W_{\pm} (z, \cdot \, ,c)) 
= C_{\pm, 1}^* C_{\pm, 2} - C_{\pm, 2}^* C_{\pm, 1}.     \lb{3.32} 
\end{align}

At this point we introduce the notion of matrix-valued principal solutions of $L U = \lambda U$, 
$\lambda \in \bbR$.

%%%%%%%%%%%%
\begin{definition} \lb{d3.4}
Assume Hypothesis \ref{h3.1} and let $\lambda \in \bbR$. \\ 
$(i)$ Suppose that $U_a(\lambda, \cdot \,)$ is a self-conjugate 
solution of $L U = \lambda U$ that is invertible on $(a,c]$ 
for some $c \in (a,b)$. Then $U_a(\lambda,\cdot \,)$ is called a {\it principal solution} 
of $L U = \lambda U$ at $a$ if 
\begin{equation}
\lim_{x \downarrow a} \bigg[\int_x^c dx' \, 
U_a(\lambda,x')^{-1} P(x')^{-1} \big[U_a(\lambda,x')^{-1}\big]^*\bigg]^{-1} = 0.
\end{equation}   
$(ii)$ Suppose that $U_b(\lambda, \cdot \,)$ is a self-conjugate 
solution of $L U = \lambda U$ that is invertible on $[c,b)$ 
for some $c \in (a,b)$. Then $U_b(\lambda,\cdot \,)$ is called a {\it principal solution} 
of $L U = \lambda U$ at $b$ if 
\begin{equation}
\lim_{x \uparrow b} \bigg[\int_c^x dx' \, 
U_b(\lambda,x')^{-1} P(x')^{-1} \big[U_b(\lambda,x')^{-1}\big]^*\bigg]^{-1} = 0.
\end{equation}   
\end{definition} 
%%%%%%%%%%%%

Principal solutions, if they exist, are unique up to right multiplication with invertible constant 
$m \times m$ matrices:

%%%%%%%%%%%%
\begin{lemma} \lb{l3.5}
Assume Hypothesis \ref{h3.1} and let $\lambda \in \bbR$. Then if a principal solution 
$U_a(\lambda,\cdot \,)$ at $a$ $($resp., $U_b(\lambda,\cdot \,)$ at $b$$)$ of 
$L U = \lambda U$ exists, 
it is unique up to right multiplication with an invertible $($generally, $\lambda$-dependent\,$)$ constant $m \times m$ matrix.
\end{lemma}
%%%%%%%%%%%%

This follows from \cite[Theorem\ 2.3]{Co71}, or \cite[Theorem\ 10.5\,$(ii)$]{Ha82}. 

%%%%%%%%%%%%
\begin{lemma} \lb{l3.6}
Assume Hypothesis \ref{h3.1} and let $\lambda \in \bbR$. Suppose that $U_0(\lambda,\cdot \,)$ 
is a self-conjugate solution of $L U = \lambda U$ that is invertible on $(a,c]$ $($resp., $[c,b)$$)$ 
and let $V(\lambda,\cdot \,)$ be any $m \times m$ matrix-valued solution of $L U = \lambda U$. 
Then $U_0(\lambda,\cdot \,)$ is a principal solution at $b$ $($resp., $a$$)$ and 
\begin{equation}
W(U_0(\lambda, \cdot \,)^*, V(\lambda,\cdot \,)) \, \text{ is invertible,}      \lb{3.35}
\end{equation}
if and only if $V(\lambda,\cdot \,)$ is invertible near $a$ $($resp., $b$$)$ and 
\begin{equation}
\lim_{x \downarrow a} V(\lambda,x)^{-1} U_0(\lambda,x) = 0 \; 
\text{ $($resp., $\lim_{x \uparrow b} V(\lambda,x)^{-1} U_0(\lambda,x) = 0$$)$.}  \lb{3.36} 
\end{equation}
If \eqref{3.36} holds, then, for appropriate $c_a, c_b \in (a,b)$,
\begin{align}
& \bigg[\lim_{x \downarrow a} \int_x^{c_a} dx' \, V(\lambda,x')^{-1} P(x')^{-1} 
\big[V(\lambda,x')^{-1}\big]^*\bigg]^{-1}  \, 
\text{ exists and is invertible}      \lb{3.37} \\ 
& \bigg(\text{resp., $\lim_{x \uparrow b} \bigg[\int_{c_b}^x dx' \, V(\lambda,x')^{-1} P(x')^{-1} \big[V(\lambda,x')^{-1}\big]^* \bigg]^{-1}$ exists and is invertible}\bigg).      \no 
\end{align}
\end{lemma}
%%%%%%%%%%%%

Again, this follows from \cite[Proposition\ 2.4]{Co71}, or \cite[Theorem\ 10.5\,$(iii)$]{Ha82}. 

%%%%%%%%%%%%
\begin{definition} \lb{d3.7}
Assume Hypothesis \ref{h3.1}. The equation $Lu = zu$, $z \in \bbC$, is called {\it disconjugate} 
on the interval $J \subseteq (a,b)$ if every nontrivial $($ie., not identically vanishing\,$)$ solution 
$v$ of $Lu = zu$ vanishes at most once on $J$. 
\end{definition} 
%%%%%%%%%%%%

It is well known (cf., e.g., \cite[Sect.\ 2.1]{Co71}, \cite[Theorem\ XI.10.1]{Ha82}) that $Lu = zu$, 
$z\in\bbC$, is disconjugate on $J \subseteq (a,b)$ if and only if
\begin{align}
\begin{split} 
& \text{for any $x_j \in J$ and any $\eta_j \in \bbC^m$, $Lu = zu$ has a unique solution 
$v(z, \cdot \,)$}   \\
& \quad \text{satisfying $v(z,x_j) = \eta_j$, $j=1,2$.} 
\end{split} 
\end{align}
Equivalently, $Lu = zu$, $z\in\bbC$, is disconjugate on $J \subseteq (a,b)$ if and only if
\begin{align}
\begin{split} 
& \text{for any $x_j \in J$, $j=1,2$, $x_1 \neq x_2$, $Lu = zu$ has a no nontrivial solution}   \\
& \quad \text{$v(z, \cdot \,)$ satisfying $v(z,x_j) = 0$, $j=1,2$.} 
\end{split} 
\end{align}

We also recall the following useful result:

%%%%%%%%%%%%
\begin{theorem}  [\cite{Co71}, Sect.\ 2.1--2.2, 
\cite{Ha82}, Theorem\ XI.10.2] \lb{t3.8}  Assume Hypothesis \ref{h3.1}. \\
$(i)$ If $J$ is a closed half-line $($i.e., $J=[c,b)$ or $J=(a,c]$ for some $c \in (a,b)$$)$, then  
$Lu =z u$, $z\in\bbC$, is disconjugate on $J$ if and only if there exists a self-conjugate solution 
$U(z, \cdot \,)$ of $LU =z U$ such that $U(z,\cdot\,)$ is invertible on the interior of $J$. \\[1mm]
$(ii)$ If $J$ is a closed bounded subinterval of $(a,b)$ or $J \subseteq (a,b)$ is an open interval, then  
$Lu =z u$, $z\in\bbC$, is disconjugate on $J$ if and only if there exists a self-conjugate solution 
$U(z, \cdot \,)$ of $LU =z U$ such that $U(z,\cdot\,)$ is invertible on $J$. 
\end{theorem}
%%%%%%%%%%%%

Next, we derive the analog of \eqref{2.44}--\eqref{2.47} in the present matrix-valued context.

Combining \eqref{3.13}--\eqref{3.14c} and \eqref{3.26}--\eqref{3.28} yields the analog of 
\eqref{2.44}, 
\begin{align}
\begin{split}
\Phi(z,x,x_0) = \Theta(z,x,x_0) \int_{x_0}^x dx' \, \Theta(z,x',x_0)^{-1} P(x')^{-1} 
\big[\Theta({\ol z},x',x_0)^{-1}\big]^*,&   \\
z \in \bbC\backslash\bbR, \; x, x_0 \in (a,b).     \lb{3.43}
\end{split} 
\end{align}
Moreover, assuming that $L$ is in the limit point case at $a$ (resp., $b$), it is known 
(cf., \cite{HS81}) that the analog of \eqref{2.42} also holds in the form, 
\begin{align}
\begin{split} 
& M_-(z,x_0) = - \lim_{x \downarrow a} \Phi(z,x,x_0)^{-1} \Theta(z,x,x_0), \quad 
z \in \bbC \backslash \bbR, \\
& \big(\text{resp.,} M_+(z,x_0) = - \lim_{x \uparrow b} \Phi(z,x,x_0)^{-1} \Theta(z,x,x_0), \,  
\quad z \in \bbC \backslash \bbR\big).    \lb{3.44}
\end{split} 
\end{align}

Hence, we obtain the following formulas for $M_{\pm}(\,\cdot\, , x_0)$, the analogs of \eqref{2.45} and 
\eqref{2.47}:

%%%%%%%%%%%%
\begin{theorem} \lb{t3.9}
Assume Hypothesis \ref{h3.1} and suppose that $L$ is in the limit point case at $a$ 
$($resp., $b$$)$. Then
\begin{align}
& M_-(z,x_0) = \lim_{x \downarrow a} \bigg[\int_x^{x_0} dx' \, 
\Theta(z,x',x_0)^{-1} P(x')^{-1} \big[\Theta({\ol z},x',x_0)^{-1}\big]^*\bigg]^{-1},  
\quad z \in \bbC \backslash \bbR,  \lb{3.45} \\
& \bigg(\text{resp., } \, M_+(z,x_0) = - \lim_{x \uparrow b} \bigg[\int_{x_0}^x dx' \, \Theta(z,x',x_0)^{-1} P(x')^{-1} 
\big[\Theta({\ol z},x',x_0)^{-1}\big]^*\bigg]^{-1},     \no \\
& \hspace*{9.5cm} z \in \bbC \backslash \bbR\bigg).   \lb{3.46}
\end{align}
If, in addition, $Lu = \lambda_0 u$ is disconjugate on $(a,x_0]$ $($resp., $[x_0,b)$$)$ for some 
$\lambda_a \in \bbR$ $($resp., $\lambda_b \in \bbR$$)$, then 
\begin{align}
& M_-(\lambda,x_0) = \bigg[\int_a^{x_0} dx' \, \Theta(\lambda,x',x_0)^{-1} P(x')^{-1} 
\big[\Theta(\lambda,x',x_0)^{-1}\big]^*\bigg]^{-1}, \quad \lambda < \lambda_a,     \lb{3.47} \\
& \bigg(\text{resp., } \,  M_+(\lambda,x_0) = - \bigg[\int_{x_0}^b dx' \, \Theta(\lambda,x',x_0)^{-1} P(x')^{-1} 
\big[\Theta(\lambda,x',x_0)^{-1}\big]^*\bigg]^{-1},    \no\\ 
& \hspace*{9.5cm} \lambda < \lambda_b\bigg),     \lb{3.48}
\end{align} 
and 
\begin{align} 
& \big(\xi, M_- (\lambda,x_0)^{-1} \eta\big)_{\bbC^m} = \int_a^{x_0} dx' \, 
\big(\xi, \Theta(\lambda,x',x_0)^{-1} P(x')^{-1} 
\big[\Theta(\lambda,x',x_0)^{-1}\big]^* \eta\big)_{\bbC^m},   \no \\ 
& \hspace*{8.2cm} \lambda < \lambda_a, \; \xi, \eta \in \bbC^m,    \lb{3.49} \\
& \bigg(\text{resp., } \, \big(\xi, M_+ (\lambda,x_0)^{-1} \eta\big)_{\bbC^m} = - \int_{x_0}^b dx' \, 
\big(\xi, \Theta(\lambda,x',x_0)^{-1} P(x')^{-1}    \no \\ 
& \hspace*{3.5cm} \times \big[\Theta(\lambda,x',x_0)^{-1}\big]^* \eta\big)_{\bbC^m},   \quad 
\lambda < \lambda_b, \; \xi, \eta \in \bbC^m\bigg),    \lb{3.50}
\end{align}  
exists as a Lebesgue integral. 
\end{theorem}
%%%%%%%%%%%%
\begin{proof}
It suffices to focus on the endpoint $b$. Combining \eqref{3.43} and \eqref{3.44} (employing that 
$L$ is l.p.c. at $b$) yields relation \eqref{3.46}. 

For the remainder of this proof we thus assume that $Lu = \lambda_b u$ is disconjugate on 
$[x_0,b)$ for some $\lambda_b \in \bbR$ (in addition to $L$ being l.p.c. at $b$). Then,  
\begin{equation} 
\int_{x_0}^x dx' \, \Theta(\lambda,x',x_0)^{-1} P(x')^{-1} 
\big[\Theta(\lambda,x',x_0)^{-1}\big]^*, \quad \lambda < \lambda_b,     \lb{3.51}
\end{equation}
is strictly monotone increasing with respect to $x > x_0$. Recalling the well-known fact 
(cf.\ \cite[Lemma\ 2.1]{AN76}), 
\begin{align}
\begin{split}
& \text{If $0 \leq C_1 \leq C_2 \leq \cdots \leq C_\infty$, with
$C_n, C_\infty \in \cB_{\infty}(\cH)$, $n\in\bbN$,}      \lb{3.52} \\
& \quad \text{then $\lim_{n\to\infty} \|C_n - C\|_{\cB(\cH)} = 0$ for
some $C\in\cB_{\infty}(\cH)$,}
\end{split}
\end{align}
one infers convergence of the $m \times m$ matrix $- \int_{x_0}^x dx' \, \cdots$ to 
$- \int_{x_0}^b dx' \, \cdots$ on the right-hand-side in \eqref{3.48} as $x \uparrow b$. In addition, 
the monotone convergence theorem implies the existence of 
\begin{equation}
\int_{x_0}^b dx' \, 
\big(\xi, \Theta(\lambda,x',x_0)^{-1} P(x')^{-1} 
\big[\Theta(\lambda,x',x_0)^{-1}\big]^* \xi\big)_{\bbC^m},  \quad  
 \lambda < \lambda_b, \; \xi \in \bbC^m,    \lb{3.53} 
\end{equation} 
as a Lebesgue integral. The general case depicted in \eqref{3.50} for $\xi, \eta \in \bbC^m$ 
then follows by polarization. 

It remains to prove equality of $M_+(\lambda,x_0)$ with the right-hand side of \eqref{3.48} for 
$\lambda < \lambda_b$. We start by noting that disconjugacy of $Lu = \lambda_b u$ implies analyticity of 
$M_+(\, \cdot \, ,x_0)$ on $\bbC \backslash [\lambda_b, \infty)$ and hence the fact 
that the $m \times m$ matrix-valued measure $\Omega(\, \cdot \, , x_0)$ in the Nevanlinna--Herglotz representation for $M_+(\, \cdot \, ,x_0)$ is supported on $[\lambda_b, \infty)$, that is, one infers 
the representation,
\begin{align}
& M_+(z,x_0) = A + \int_{[\lambda_b, \infty)} d \Omega(\lambda, x_0)  
\big[(\lambda - z)^{-1} - \lambda (1 + \lambda^2)^{-1}\big], \quad z \in \bbC \backslash [\lambda_b, \infty),
\no \\
& A = A^* \in \bbC^{m \times m}, \quad 
\int_{[\lambda_b, \infty)} d (\xi, \Omega(\lambda, x_0) \xi)_{\bbC^m} (1 + \lambda^2)^{-1} < \infty, \; 
\xi \in \bbC^m.   \lb{3.54} 
\end{align}
Similarly, one infers that for each $x \in [x_0,b)$ the $m \times m$ matrix-valued function, 
$M_{+,x}(\, \cdot \, ,x_0)$, defined by 
\begin{align} 
& M_{+,x}(z,x_0) := - \Phi(z,x,x_0)^{-1} \Theta(z,x,x_0)     \no \\
& \quad = - \bigg[\int_{x_0}^x dx' \, \Theta(z,x',x_0)^{-1} P(x')^{-1}   
\big[\Theta({\ol z},x',x_0)^{-1}\big]^*\bigg]^{-1},    \lb{3.55} \\
& \hspace*{4.9cm} z \in \bbC \backslash [\lambda_b, \infty), \; x \in (x_0,b),    \no 
\end{align} 
is meromorphic on $\bbC$ and also analytic on $\bbC \backslash [\lambda_b, \infty)$. 
General Weyl--Titchmarsh theory in connection with the interval $[x_0,x]$, $x \in (x_0,b)$, 
where $x_0, x$ are regular endpoints for $L$, yields that for fixed $x_0, y \in (a,b)$, 
$M_{+,y}(\, \cdot \, ,x_0)$, and hence the $m \times m$ matrix-valued 
integral in \eqref{3.55}, represents a matrix-valued meromorphic Herglotz--Nevanlinna  
function (cf.\ \cite{HS81}). Indeed, employing \eqref{3.14d} yields
\begin{equation}
M_{+,y} (z,x_0)^* = M_{+,y}({\ol z}, x_0), \quad z \in \bbC \backslash \bbR, 
\end{equation}
and introducing 
\begin{align}
\begin{split} 
U_y(z,x,x_0) = \Theta(z,x,x_0) + \Phi(z,x,x_0) [- \Phi(z,y,x_0)^{-1} \Theta(z,y,x_0)],& \\ 
z\in \bbC \backslash \bbR, \; x \in [x_0, y], y \in (x_0,b),&      \lb{3.56} 
\end{split} 
\end{align}
a combination of \eqref{3.11} (for $z = {\ol z_1} = z_2$), $U_y(z,y,x_0) = 0$, \eqref{3.14}--\eqref{3.14c} 
imply the identity
\begin{align}
\begin{split} 
& \Im(M_{+,y}(z,x_0)) = \Im \big(- \Phi(z,y,x_0)^{-1} \Theta(z,y,x_0)\big)   \\
& \quad = \Im(z) \int_{x_0}^y dx' \, U_y(z,x',x_0)^* R(x') U_y(z,x',x_0),   
\quad z\in \bbC \backslash \bbR, \; y \in (x_0,b).     \lb{3.57} 
\end{split} 
\end{align}
Again, disconjugacy of $Lu = \lambda_b u$ implies that the $m \times m$ matrix-valued measure 
$\Omega_x$ associated with $M_{+,x}(\, \cdot \, ,x_0)$ in  \eqref{3.55}, is again supported on 
$[\lambda_b, \infty)$, that is, for each $x \in (x_0,b)$, 
\begin{align}  
& M_{+,x}(z,x_0) = - \Phi(z,x,x_0)^{-1} \Theta(z,x,x_0)   \no \\ 
& \hspace*{1.8cm}= - \bigg[\int_{x_0}^x dx' \, \Theta(z,x',x_0)^{-1} P(x')^{-1} 
\big[\Theta({\ol z},x',x_0)^{-1}\big]^*\bigg]^{-1}     \no \\
& \quad = A_x + \int_{[\lambda_b, \infty)} d \Omega_x(\lambda, x_0)  
\big[(\lambda - z)^{-1} - \lambda (1 + \lambda^2)^{-1}\big], \quad z \in \bbC \backslash [\lambda_b, \infty),  
\lb{3.58} \\
& A_x = A_x^* \in \bbC^{m \times m}, \quad 
\int_{[\lambda_b, \infty)} d (\xi, \Omega_x (\lambda, x_0) \xi)_{\bbC^m} (1 + \lambda^2)^{-1} 
< \infty, \; \xi \in \bbC^m.     \no  
\end{align} 
In accordance with the limiting relation \eqref{3.44}, the finite measures 
$d \Omega_x(\lambda,x_0) (1 + \lambda^2)^{-1}$ 
converge to $d \Omega (\lambda,x_0)(1 + \lambda^2)^{-1}$ as $x \uparrow b$ in the 
weak-${}^*$ sense (cf. also \cite{KR01}), that is,  
\begin{equation} 
\lim_{x\uparrow 0} \int_{[\lambda_b, \infty)} d \Omega_x(\lambda, x_0)  
\, (1 + \lambda^2)^{-1} f(\lambda) = \int_{[\lambda_b, \infty)} 
d \Omega(\lambda, x_0) (1 + \lambda^2)^{-1} f(\lambda)     \lb{3.59} 
\end{equation} 
for all $f \in C(\bbR) \cap L^{\infty}(\bbR; d\lambda)$. 
The Nevanlinna--Herglotz representation \eqref{3.58} for $M_{+,x}(\, \cdot \,,x_0)$ demonstrates 
that for any compact $K \subset \bbC\backslash [\lambda_b,\infty)$, there exists a constant 
$C(K)>0$ such that $\|M_{+,x}(z,x_0)\|_{\cB(\bbC^m)} \leq C(K)$ uniformly with respect to 
$z \in K$ and $x \in (x_0,b)$. An application of Vitali's Theorem (see, e.g., 
\cite[Sect.\ 7.3]{Re98}) then proves that the convergence in \eqref{3.44} extends to 
\begin{equation}
M_+(z,x_0) = - \lim_{x \uparrow b} \Phi(z,x,x_0)^{-1} \Theta(z,x,x_0), \quad 
z \in \bbC \backslash [\lambda_b,\infty),    \lb{3.60}
\end{equation}
in particular, it applies to $z < \lambda_b$ and hence yields \eqref{3.48}.
\end{proof}
%%%%%%%%%%%%

We are not aware of any source containing formulas of the type \eqref{3.45}--\eqref{3.50}.
Naturally, these formulas extend to the more general self-adjoint boundary conditions at the regular 
endpoint $x_0 \in (a,b)$ discussed in detail in \cite{HS81} (cf.\ also \cite{CG01}, \cite{CG02}) in 
the matrix-valued context, extending the scalar case described in \eqref{2.48}--\eqref{2.50}. We 
omit further details at this point. 

The main result of this section then reads as follows.

%%%%%%%%%%%%
\begin{theorem} \lb{t3.10}
Assume Hypothesis \ref{h3.1}. \\ 
$(i)$ Suppose that for some $\lambda_a \in \bbR$, 
$(u,[T_{min} - \lambda_a I] u)_{L^2((a,b); R dx; \bbC^m)}\geq 0$ for all $u \in \dom(T_{min})$ with 
$u=0$ in a neigborhood of $b$. In addition, assume that $L$ is in the limit point case at $a$. 
Then for all $\lambda < \lambda_a$, the Weyl--Titchmarsh solution $\Psi_-(\lambda,\cdot \,)$ is 
also a principal solution of $L U = \lambda U$ at $a$, that is, for $x, x_0$ to the left of the first 
zero of $\det_{\bbC^m}(\Psi_- (\lambda, \cdot \,)), \det_{\bbC^m}(U_a(\lambda,\cdot \,))$ 
$($if any\,$)$,  
\begin{equation} 
\Psi_-(\lambda,x) \Psi_-(\lambda,x_0)^{-1} = U_a(\lambda,x) U_a(\lambda,x_0)^{-1}.    
\end{equation}   
$(ii)$ Suppose that for some $\lambda_b \in \bbR$, 
$(u,[T_{min} - \lambda_b I] u)_{L^2((a,b); R dx; \bbC^m)}\geq 0$ for all $u \in \dom(T_{min})$ with 
$u=0$ in a neigborhood of $a$. In addition, assume that $L$ is in the limit point case at $b$. 
Then for all $\lambda < \lambda_b$, the Weyl--Titchmarsh solution $\Psi_+(\lambda,\cdot \,)$ is 
also a principal solution of $L U = \lambda U$ at $b$, that is, for $x, x_0$ to the right of the last 
zero of $\det_{\bbC^m}(\Psi_+ (\lambda, \cdot \,)), \det_{\bbC^m}(U_b(\lambda,\cdot \,))$ 
$($if any\,$)$,  
\begin{equation} 
\Psi_+(\lambda,x) \Psi_+(\lambda,x_0)^{-1} = U_b(\lambda,x) U_b(\lambda,x_0)^{-1}.    
\end{equation} 
\end{theorem}
%%%%%%%%%%%%
\begin{proof}
It suffices to consider item $(ii)$. By \cite[Theorem\ XI.10.3 ]{Ha82}, the assumption on 
$T_{min} - \lambda_b I$ implies that for all $\lambda < \lambda_b$ and all $c \in (a,b)$, 
$Lu =\lambda u$ is disconjugate on $[c,b)$. By \cite[Theorem 2.3]{Co71} or 
\cite[Theorem\ XI.10.5]{Ha82}, $LU = \lambda U$ has a principal solution $U_b(\lambda, \cdot\,)$ 
for all $\lambda < \lambda_b$. Without loss of generality we may uniquely determine 
$U_b(\lambda, \cdot \, ,x_0)$ by demanding the normalization $U_b(\lambda, x_0; x_0) = I_m$.  
As proved in \cite[p.\ 44--45]{Co71}, it is possible to approximate $U_b(\lambda, \cdot \, ,x_0)$ as 
follows: For $y \in (x_0,b)$, consider the unique solution $U_y(\lambda,\cdot;x_0)$, of 
$LU = \lambda U$, $\lambda < \lambda_b$, satisfying
\begin{equation}
U_y (\lambda,x_0,x_0) = I_m, \quad U_y(\lambda,y,x_0) = 0.   \lb{3.61} 
\end{equation}
Then
\begin{equation}
U_b(\lambda, \cdot \, ,x_0) = \lim_{y \uparrow b} U_y(\lambda,\cdot,x_0), \quad \lambda < \lambda_b. 
\lb{3.62} 
\end{equation}
In addition, one obtains that
\begin{equation}
U_y (\lambda, \cdot \, , x_0) = \Theta (\lambda, \cdot \, ,x_0) + \Phi (\lambda, \cdot \, ,x_0) 
M_{+,y} (\lambda; x_0),     \lb{3.63}
\end{equation}
with $M_{+,y} (\lambda; x_0)$ introduced in \eqref{3.55}. Employing the convergence result 
\eqref{3.47}, that is, $\lim_{y \uparrow b} M_{+,y} (\lambda,x_0) = M_+(\lambda,x_0)$, 
$\lambda < \lambda_b$, in \eqref{3.63} thus also yields 
\begin{equation}
\lim_{y \uparrow b} U_y (\lambda, \cdot \, , x_0) = \Psi_+(\lambda,x) \Psi_+(\lambda,x_0)^{-1}, 
\quad \lambda < \lambda_b.     \lb{3.64}
\end{equation}
A comparison of \eqref{3.62} and \eqref{3.64} then proves
\begin{equation}
U_b(\lambda, \cdot \, ,x_0) = \Psi_+(\lambda,x) \Psi_+(\lambda,x_0)^{-1}, 
\quad \lambda < \lambda_b,     \lb{3.65} 
\end{equation} 
completing the proof. 
\end{proof}
%%%%%%%%%%%%

We emphasize that the continuity assumptions on the coefficients in $L$ made in the context 
of oscillaton theory in \cite[Sect.\ 2.1]{Co71}, \cite[Sect.\ XI.10]{Ha82} are not necessary and the 
quoted results in this section all extend to our current Hypothesis \ref{h3.1}.

We also note that while we focused on Sturm--Liouville operators with matrix-valued coefficients, 
a treatment of more general singular Hamiltonian systems (along the lines of \cite{CG02}, 
\cite[Ch.\ 2]{Co71}, \cite{HS93}--\cite{HS86}, \cite{Re58}, \cite{Re64}, \cite[Ch.\ VII]{Re71}, 
\cite[Chs.\ V, VI]{Re80}, is clearly possible.

Emboldened by the results in Theorem \ref{t3.10} in the matrix context, one might guess that 
if $T_{min} \geq \lambda_0 I$ for some $\lambda_0 \in \bbR$, positivity of the solution 
$u(\lambda_0, \cdot \,)$ of $\ell u=\lambda_0 u$, or alternatively, $u\neq 0$ in
Theorem \ref{t2.8} could be translated to the matrix-valued case in a multitude of different ways. 
Let $U(\lambda_0, \cdot \,) \in \bbC^{m\times m}$ denote a matrix-valued solution of 
$L U=\lambda_0 U$, then here is a possible list of ``positivity results'' 
one could imagine in the matrix context from the outset: 

\medskip 

\noindent $(I)$ \hspace{2.6mm} $U\in\bbC^{m\times m}$ is invertible. \\[1mm] 
\noindent  $(II)$ \hspace*{.6mm} $U\in\bbC^{m\times m}$ is positive definite. \\[1mm]
\noindent  $(III)$  $U\in\bbC^{m\times m}$ is positivity preserving. \\[1mm]
\noindent $(IV)$  $U\in\bbC^{m\times m}$ is positivity improving.  

\medskip 

For completenes we briefly recall the notions of positivity preserving (resp., improving) 
matrices:   

%%%%%%%%%%%%%% 
\begin{definition} \lb{d3.11} 
Let $A=\big(A_{j,k}\big)_{1\leq j,k\leq m}\in\bbR^{m\times m}$ for some
$m\in\bbN$. \\
$(i)$ $A$ is called {\it positivity preserving} if $A_{j,k}\geq 0$ for all
$1\leq j,k\leq m$.  \\[1mm]  
$(ii)$ $A$ is called {\it positivity improving} if $A_{j,k}> 0$ for all $1\leq j,k\leq m$.
\end{definition}
%%%%%%%%%%%%%%

However, item $(II)$ implies self-adjointness of 
$U(\lambda_0, \cdot \,)$ and hence upon invoking the equation adjoint to 
$L U=\lambda_0 U$, commutativity of $U(\lambda_0, \cdot \,)$ and 
$Q(\cdot)$.  Our next example, a matrix-valued Schr\"odinger 
operator (i.e., $P(\cdot) = R(\cdot) = I_m$ in $L$), provides a simple counter-example to positive definiteness.

%%%%%%%%%%%%%%
\begin{example}\lb{e3.13}
Let $m=2$, $(a,b)=\bbR$, $P(\cdot) = R(\cdot) = I_2$ a.e. on $\bbR$, and 
\begin{equation}
Q(x) = \begin{pmatrix} 0 & 1 \\ 1 &  2\end{pmatrix}
\end{equation}
in $L$.  One verifies that
\begin{equation}
T_{min}\geq -2.
\end{equation}
Taking $E=-2$, the general solution to $LU=-2U$ has the form
\begin{equation}\lb{3.77}
U(x)= \begin{pmatrix} U_{1,1}(x) & U_{1,2}(x) \\
U_{2,1}(x) & U_{2,2}(x)
\end{pmatrix},\quad x\in \bbR,
\end{equation}
where 
\begin{align}
U_{1,1}(x) &= c_1e^{\sqrt{2}x} + c_2e^{-\sqrt{2}x},\lb{3.78}\\
U_{2,1}(x) &= \widetilde{c}_1 e^{2x} + \widetilde{c}_2e^{-2x} - (c_1/2)e^{\sqrt{2}x} - (c_2/2)e^{-\sqrt{2}x},\\
U_{2,2}(x) & = d_1\cos(x)+d_2\sin(x)+d_3e^{\sqrt{7}x}+d_4e^{-\sqrt{7}x},\\
U_{1,2}(x) &= \widetilde{d_1}e^{\sqrt{2}x} + \widetilde{d}_2e^{-\sqrt{2}x} - (d_1/3)\cos(x) - (d_2/3)\sin(x)\\
&\quad + (d_3/5)e^{\sqrt{7}x} + (d_4/5)e^{-\sqrt{7}x},\lb{3.82}
\end{align}
and $c_j,\widetilde{c}_j, \widetilde{d}_j$, $j\in\{1,2\}$, and $d_k$, $k\in \{1,2,3,4\}$ are arbitrary parameters.  No solution of the form \eqref{3.77}--\eqref{3.82} is positive definite for all $x\in \bbR$.  If such a solution were positive definite, it would commute with $Q$, so it suffices to show that solutions that commute with $Q$ are not positive definite.  By writing out $UQ=QU$, and equating corresponding matrix entries, one infers that $U$ commutes with $Q$ if and only if 
\begin{equation}\lb{3.83}
\widetilde{c}_1 = \widetilde{c}_2 = d_1 = d_2 = d_3 = d_4 = 0,\, c_1=-2\widetilde{d}_1,\, c_2=-2\widetilde{d}_2,
\end{equation}
with $\widetilde{d}_1$ and $\widetilde{d}_2$ arbitrary.  $($One could just as well arrive at \eqref{3.83} using self-adjointness of $U(x)$.$)$  Taking \eqref{3.83} for granted, for a fixed choice of constants $\widetilde{d}_1$ and $\widetilde{d}_2$, $U(x)$ has the form
\begin{equation}\lb{3.84}
U(x) = \begin{pmatrix} -2A(x) & A(x)\\
A(x) & 0
\end{pmatrix},\quad x\in \bbR,
\end{equation}
where we have set
\begin{equation}
A(x) = \widetilde{d}_1 e^{\sqrt{2}x} + \widetilde{d}_2 e^{-\sqrt{2}x},\quad x\in \bbR.
\end{equation}
One then computes the eigenvalues of $U(x)$ in \eqref{3.84} to be
\begin{equation}
\lambda_{\pm}(x) = -A(x) \pm \sqrt{2}|A(x)|,\quad x\in \bbR.
\end{equation}
Since $\lambda_-(x)\leq 0$ for all $x\in \bbR$, $U(x)$ is not positive definite for any value of $x$, let alone for all $x\in \bbR$.
\end{example}
%%%%%%%%%%%%%%

In addition, items $(III)$ and $(IV)$ are ruled out by the following elementary 
constant coefficient example:

%%%%%%%%%%%%%%
\begin{example} \lb{e3.12}
Let $m=2$, $(a,b)=\bbR$, $P(\cdot) = R(\cdot) = I_2$ a.e.\ on $\bbR$, $q_0\in\bbR\backslash\{0\}$, and 
\begin{equation}\lb{4.19}
Q(x) = \begin{pmatrix} 0 & q_0 \\ q_0 &  0\end{pmatrix}. 
\end{equation}
One verifies that 
\begin{equation}
T_{min} \geq - |q_0|. 
\end{equation}
Assuming that $E\leq -|q_0| <0$, let 
\begin{equation}
\delta_{\pm} (E) = \sqrt{|E\pm q_0|}.
\end{equation}
We claim that $U_{\infty} (E, \cdot \,)$, defined by 
\begin{equation}
U_{\infty} (E,x)=\begin{pmatrix} e^{-\delta_- (E) x} 
& -e^{-\delta_+ (E) x} \\
e^{-\delta_- (E) x} & e^{-\delta_+ (E) x} \end{pmatrix}, \quad
E\leq -|q_0|<0, \; x \in \bbR,     \lb{4.24}
\end{equation}
is a  principal solution of $L U = E U$ at $\infty$.
That $U_{\infty}(E, \cdot \,)$ is self-conjugate follows from the observation that 
\begin{equation}
(U_{\infty}(E,x)')^*U_{\infty}(E,x) =\begin{pmatrix}
-2\delta_- (E) e^{-\delta_- (E) x}&
0 \\  0 & -2\delta_+ (E) e^{-\delta_+ (E) x}  \end{pmatrix}. 
\end{equation}
Since $\det(U_{\infty}(E,x)) = 2e^{-(\delta_- (E)+\delta_+ (E))x}$,
one infers that $U_{\infty}(E, \cdot \,)$  is invertible on $\bbR$. That this
particular solution is  principal at $\infty$ follows from the fact that 
\begin{align}
\begin{split} 
& \bigg[\int_0^x\, dx' \, U_{\infty}(E,x')^{-1}\big[U_{\infty}(E,x')^{-1}\big]^*\bigg]^{-1}    \\
& \quad =\begin{pmatrix} 4\delta_- (E)/(e^{2\delta_- (E) x}-1) & 0\\ 
0 &  4\delta_+ (E)/(e^{2\delta_+ (E) x}-1) 
\end{pmatrix}\underset{x\uparrow \infty}{\rightarrow} 0.
\end{split} 
\end{align} 
Next, we turn to all principal solutions of this example and hence
consider
\begin{align}
\widetilde U_\infty(E,x)&=\begin{pmatrix} e^{-\delta_- (E) x} 
& -e^{-\delta_+ (E) x} \\
e^{-\delta_- (E) x} & e^{-\delta_+ (E) x} \end{pmatrix}
\begin{pmatrix} \alpha & \beta \\
\gamma & \epsilon \end{pmatrix} \no\\
&=\begin{pmatrix} \alpha e^{-\delta_- (E) x}-\gamma e^{-\delta_+ (E)x}
& \beta e^{-\delta(E)x}-\epsilon e^{-\delta_+ (E) x} \\
\alpha e^{-\delta_- (E) x}+\gamma e^{-\delta_+ (E)x} & 
\beta e^{-\delta(E)x} +\epsilon e^{-\delta_+ (E)x}\end{pmatrix}
\end{align}
with $\left(\begin{smallmatrix} \alpha & \beta \\
\gamma & \epsilon \end{smallmatrix}\right)\in\bbC^{2\times 2}$ a
nonsingular constant matrix. 

\noindent 
By inspection, $\widetilde U_\infty(-|q_0|, \cdot \,)$ is never positivity
preserving $($let alone, improving\,$)$. 
\end{example}
%%%%%%%%%%%%%

The question of positive vector solutions of $L u = \lambda_0 u$ has been studied in the literature 
and we refer, for instance to \cite{Ah83}, \cite{AL76}, \cite{AL77}, \cite{AL78b}, \cite{FZ95}, 
\cite{SS78}.

We conclude with the remark that the results presented in this section extend from the case of 
$m \times m$ matrix-valued coefficients to the situation of operator-valued coefficients in an 
infinite-dimensional, complex, separable Hilbert space. For instance, basic Weyl--Titchmarsh 
theory for the infinite-dimensional case has been derived by Gorbachuk \cite{Go68}, 
Gesztesy, Weikard, and Zinchenko \cite{GWZ13}, \cite{GWZ13a}, Saito \cite{Sa71}, \cite{Sa71a}, 
\cite{Sa72}, \cite{Sa79} (see also \cite{GKMT01}, \cite[Chs.\ 3, 4]{GG91}, \cite{MN11}, \cite{Mo07}, \cite{Mo09}, \cite[Chs.\ 1--4]{RK05}, \cite{Tr00}, \cite{VG70}). For oscillation theoretic results 
in the infinite-dimensional context we refer, for example, to \cite{EL80}, \cite{EP76}, \cite{EP77}, 
\cite{HH70}, \cite{KW79}, \cite{Ku82}. A detailed treatment of this circle of ideas will appear 
elsewhere. 

\medskip

%%%%%%%%%%%%%%%%%%%%%%%%%%%%%%%%%%%%%
\noindent 
{\bf Acknowledgments.} We are indebted to Don Hinton for directing our attention to reference 
\cite{HS81} and the connection between Weyl--Titchmarsh and principal solutions pointed out 
therein. We are also grateful to Shinichi Kotani for pointing out reference \cite{KK74} to us in 
connection with formula \eqref{2.47}. Finally, we are grateful to Hubert Kalf for kindly bringing reference \cite{Ku67} to our attention and for repeated correspondence in this context. 
%%%%%%%%%%%%%%%%%%%%%%%%%%%%%%%%%%%%%

%%%%%%%%%%%%%%%%%%%%%%%%%%%%%%%%
%%%%%%%%%%%%%%%%%%%%%%%%%%%%%%%%
 
\end{document}